\documentclass[12pt,a4paper]{article}

\usepackage[utf8]{inputenc}
\usepackage[english]{babel}
\usepackage{amsmath}
\usepackage{amsfonts}
\usepackage{amssymb}
\usepackage{graphicx}
\usepackage{enumerate}

\usepackage{amsthm}
\usepackage{mathrsfs}
\usepackage{proof}
\usepackage{color}

\usepackage{url}

\usepackage[left=3cm,right=2cm,top=3cm,bottom=2cm]{geometry}
\author{}
\title{Tableau systems for some {I}vlev-like (quantified) modal logics}

\author{M.E. Coniglio, L. Fari\~nas del Cerro and N.M. Peron}

\theoremstyle{theorem}
\newtheorem{teo}{Theorem}[section]
\newtheorem{lema}[teo]{Lemma}
\newtheorem{prop}[teo]{Proposition}
\newtheorem{coro}[teo]{Corollary}

\theoremstyle{definition}
\newtheorem{defi}[teo]{Definition}
\newtheorem{obs}[teo]{Remark}

\newcommand{\noi}{\noindent}

\newcommand{\F}{\mathcal{F}}

\newcommand{\tm}{\textbf{Tm}}
\newcommand{\sqm}{\textbf{S4m}}
\newcommand{\scm}{\textbf{S5m}}
\newcommand{\tms}{\textbf{Tm*}}
\newcommand{\sqms}{\textbf{S4m*}}
\newcommand{\scms}{\textbf{S5m*}}
\newcommand{\lms}{\textbf{Lm*}}

\newcommand{\laT}{\textsf{T}}
\newcommand{\lat}{\textsf{t}}
\newcommand{\laf}{\textsf{f}}
\newcommand{\laF}{\textsf{F}}
\newcommand{\laL}{\textsf{L}}
\newcommand{\lab}{{:}}
\newcommand{\lac}{\textsf{c}}

\newcommand{\axK}{\textsf{(K)}}
\newcommand{\axKu}{\textsf{(K1)}}
\newcommand{\axKd}{\textsf{(K2)}}
\newcommand{\axMu}{\textsf{(M1)}}
\newcommand{\axMd}{\textsf{(M2)}}
\newcommand{\axMt}{\textsf{(M3)}}
\newcommand{\axMc}{\textsf{(M4)}}
\newcommand{\axDNu}{\textsf{(DN1)}}
\newcommand{\axDNd}{\textsf{(DN2)}}
\newcommand{\axT}{\textsf{(T)}}
\newcommand{\axSq}{\textsf{(4)}}
\newcommand{\axSc}{\textsf{(5)}}
\newcommand{\axBF}{\textsf{(BF)}}
\newcommand{\axCBF}{\textsf{(CBF)}}
\newcommand{\axNBF}{\textsf{(NBF)}}
\newcommand{\axPBF}{\textsf{(PBF)}}

\newcommand{\axu}{\textsf{(Ax1)}}
\newcommand{\axd}{\textsf{(Ax2)}}
\newcommand{\axt}{\textsf{(Ax3)}}
\newcommand{\axq}{\textsf{(Ax4)}}
\newcommand{\axc}{\textsf{(Ax5)}}
\newcommand{\axs}{\textsf{(Ax6)}}

\newcommand{\MP}{\textsf{MP}}
\newcommand{\Gen}{\textsf{Gen}}

\DeclareSymbolFont{symbolsC}{U}{txsyc}{m}{n}
\DeclareMathSymbol{\strictif}{\mathrel}{symbolsC}{74}

\begin{document}

\maketitle

\begin{abstract}
Ivlev's pioneering work started in the 1970's showed a new and promissory way in the study of modal logic from the perspective of many-valued logics. Continuing our previous work on Ivlev-like non-normal modal logics with non-deterministic semantics, we present in this paper  tableau systems for {\bf Tm}, {\bf S4m} and {\bf S5m}, the non-normal versions of {\bf T}, {\bf S4} and {\bf S5}, respectively, as well as for their corresponding  first-order extensions \tms, \sqms\ and \scms.
\end{abstract}

\section*{Introduction}

Under a traditional perspective, we can distinguish logical operators into two major groups. In the first one, we have the most usual connectives and quantifiers, which are called \emph{extensional}. They seek to symbolically represent the meaning of certain expressions in natural language, such as: ``and'', ``or'', ``not'', ``implies'', as well as ``for all'' and ``exists". In the second group, we have the \emph{intensional} connectives, for example: ``it is necessary that'', ``it is obligatory that'', ``believes in'', and so on.

The most important difference between these logical operators, as noted by Frege, is that only in the case of extensional operators does the truth value of the complex sentence depend exclusively on the truth value of its parts. This, however, is not the case for intensional operators, which makes a formal semantic approach for them much more difficult.

In the specific case of modal logic, let us remember that, although the extensional semantics of classical logic was very well established in the mid-thirties of the twentieth century with the works of Tarski, only in the sixties the clear and intuitive semantics presented by Kripke  managed to formally represent the meaning of expressions such as ``it is necessary that'' and ``it is possible that''. Such semantics, which came to be called {\em relational semantics} or {\em possible worlds semantics}, caused a real revolution in the way we understand intensional operators. So much so that in \cite{bla:rij:ven:01}, the authors defend the slogan that modal logic is relational semantics.

Such a slogan, while warrantable because of the resounding success of relational semantics, ignores an alternative to Kripkean semantics that has been unknown for decades by much of the modal logic community, but that caught the attention of some of them in recent years. This interest is reflected in the growing number of publications on the so-called non-deterministic semantics for modal logics. Such semantics can be seen as an alternative to the Kripkean approach of capturing the meaning of intensional sentences.

This is because the requirement of the Fregean principle of extensionality --- namely, that the truth value of a complex sentence depends exclusively on the truth value of its constituent parts --- is here weakened in the following sense: instead of having a single truth-value for each complex sentence, we have non-empty sets of possible truth-values.\footnote{In technical terms, the connectives are interpreted as multiperators (or multifunctions) instead of operators (or functions).} The valuations must then choose some value among the possible ones.

As far as modal logic is concerned,  Ivlev in~\cite{ivl:73}, \cite{ivl:85}, \cite{ivl:88}, and~\cite{ivl:13} seems to have been one of the first to think of a set of truth values to capture the formal meaning of the intensional operators ``it is necessary that" and ``it is possible that''. In any case, this alternative approach to modal logic has been ignored by the community for decades, although it has had a timid but growing interest in these semantics in recent years.

In previous works, we sought to continue contributing to the development of non-deterministic semantics for modal logics. First, we presented a non-deterministic six-valued semantics for deontic operators in~\cite {con:cer:per:15} and~\cite{con:cer:per:17}. Then, we verified the viability of this type of semantics for modal systems even weaker than deontic ones, with 8 values in~\cite{con:cer:per:19}. Finally, we extended our approach to first-order modal logic with equality in~\cite{con:cer:per:21}. Closely related results were obtained independently in~\cite{omo:sku:16} and~\cite{OS:20}). New results in this subject were presented in~\cite{gratz:21b} and~\cite{paw:larosa:21} (see Section~\ref{finalsect}).

This article is organized as follows. In Section~\ref{secTm} we present the notion of non-deterministic semantics and some linguistic intuitions behind some four-valued Ivlev-like modal systems. In Section~\ref{secTm*} we present non-deterministic semantics concerning the quantified extension of these systems. In Section~\ref{secAxs} we present a list of axioms and inference rules which are used to define the Hilbert calculi for all these systems, and some of their metatheorems are stated. In Section~\ref{secTableaux} we present the analytical tableaux method for propositional and quantified versions of these systems. Finally, in the last section we compare our results with some decidability results from classical logic, many-valued logic, and Kripkean modal logics.

\section{Some four-valued Ivlev-like modal systems} \label{secTm}

Let us define first, what we will consider here as a propositional modal language. Let $\mathcal{P} = \{p_0, p_1, \ldots \}$ be an infinite denumerable set of propositional variables. The set $For_P$ of propositional modal formulas is generated as follows: (i) any element of $\mathcal{P}$ is an atomic formula; (ii) if $\varphi$ is a formula, then $(\neg \varphi)$ and $(\Box \varphi)$ are formulas; (iii) if $\varphi$ and $\psi$ are formulas, then $(\varphi \to \psi)$ is a formula; (iv) nothing else is a formula. We will omit parenthesis when the readability is unambiguous.

Ivlev's modal semantics is a generalization of the multi-valued matrices. After the seminal articles~\cite{avr:lev:01} and~\cite{avr:lev:05}, this semantic is called \emph{non-deterministic matrix semantics} (see also~\cite{avr:zam:11}).

\begin{defi} A \emph{non-deterministic matrix} (Nmatrix) for a propositional language $\mathcal{L}$ is a triple $\mathcal{M} = \langle \mathcal{V}, \mathcal{D}, \mathcal{O} \rangle$ such that:
	\begin{itemize}
		\item $\mathcal{V}$ is a non-empty set of truth values;
		\item $\mathcal{D}$ (designated truth values) is a non-empty proper subset of $\mathcal{V}$;
		\item For any n-ary connective $\#$, $\mathcal{O}$ includes a correspondent interpretation function $\tilde{\#} : \mathcal{V}^n \to \wp(\mathcal{V}) \setminus \{ \emptyset \} $
	\end{itemize}	
\end{defi}

\noi Valuations over a given non-deterministic matrix are defined in a very intuitive way.

\begin{defi}[See \cite{avr:zam:11}]  \label{valNmat} 
	Let  $\mathcal{M} = \langle \mathcal{V}, \mathcal{D}, \mathcal{O} \rangle$ be a Nmatrix over $For_P$. A {\em valuation} over $\mathcal{M}$ is a function $v: For_P \to \mathcal{V}$ such that, for every $n$-ary connective $\#$  and every $\varphi_1,\ldots,\varphi_n \in For_P$:
	$$v(\#(\varphi_1,\ldots,\varphi_n)) \in \tilde{\#}(v(\varphi_1),\ldots,v(\varphi_n))$$
\end{defi}

\noi A valuation over a Nmatrix $\mathcal{M}$ \emph{satisfies} a formula $\varphi$ iff $v(\#(\varphi_1,\ldots,\varphi_n)) \in \mathcal{D}$.
We also say  that $\varphi$ is \emph{valid} over a Nmatrix $\mathcal{M}$ iff  all the valuations satisfy $\varphi$. A Nmatrix $\mathcal{M}$ is a \emph{model} of a set $\Gamma$ of formulas iff there is a valuation that satisfy every element of $\Gamma$. Finally, $\varphi$ is a \emph{semantic consequence} of $\Gamma$ over a a Nmatrix $\mathcal{M}$ iff  every valuation that is a model of $\Gamma$ satisfies $\varphi$.

Normally, modal logic are extensions of Propositional Classical Logic {\bf CL}. Thus, it is expected that all formulas that are valid in {\bf CL} should continue being  valid over $\mathcal{M}$. Because of this, the propositional operators must respect the following clauses for any $a, b \in \mathcal{V}$:

\begin{enumerate}[(i)]
	\item	$a \in \mathcal{D}$ \ iff \ $\tilde{\neg}\,a \subseteq \mathcal{V}\setminus \mathcal{D}$		
	
	\item	$a \in \mathcal{D}$ and $b \notin \mathcal{D}$ \ iff \ $a\,\tilde{\to}\,b \subseteq \mathcal{V} \setminus\mathcal{D}$	
\end{enumerate}

\noi In order to analyze these restrictive clauses in a modal context, let us consider the four values proposed by Ivlev: 

\begin{itemize}
	\item[] \  $\laT$: necessarily true
	\item[] \  $\lat$: contingently true 
	\item[] \  $\laf$: contingently false
	\item[] \  $\laF$: necessarily false / impossible
\end{itemize}

\noi such that $\mathcal{D} = \{\laT, \lat\}$. It is clear that the restrictions (i) and (ii) above seem  very week. Take, for instance, just the operator for negation $\tilde{\neg}$. Consider the sentence:

\vspace{-0.5cm}
\begin{align}
	\tag{1} \mbox{1 plus 1 is equal to 2}
\end{align}

\noi It seems to be natural to attribute to (1) the value `necessarily true', since we are prone to accept that mathematical truths are not contingent, but necessary. But consider now the negation of (1):

\vspace{-0.5cm}
\begin{align}
	\tag{2} \mbox{1 plus 1 is not equal to 2}
\end{align}

\noi If (1) is necessarily true, then (2) should be necessarily false, that is, impossible.  To guarantee this, we must force that the negation of `necessarily true' is `necessarily false'. The reciprocal should also intuitively apply: the negation of something impossible should be necessary. An analogous requirement seems reasonable in the case of contingent propositions so that the negation of `contingently true' should be `contingently false' and vice versa.  These considerations lead us to the following truth table for the operator $\tilde{\neg}$:

\begin{displaymath}
	\def\arraystretch{1.3}
	\begin{array}{|l|l|}
		\hline  	   			& \tilde{\neg}\\
		\hline \hline  	\laT 	& \laF \\
		\hline 		   	\lat 	& \laf \\
		\hline 			\laf 	& \lat \\
		\hline 			\laF	& \laT \\
		\hline	
	\end{array}
\end{displaymath}

\

\noi The argument for constraining the operator $\tilde{\to}$  is a bit more complex. Take, for instance, the sentence:

\vspace{-0.5cm}
\begin{align}
	\tag{3} \mbox{If 1 plus 1 is equal to 2, then 2 minus 1 is equal to 1}
\end{align}

\noi Suppose (3) is necessarily true. From (2) and (3), it seems reasonable to assume that we should infer that ``2 minus 1 equals 1'' must be necessarily true. 

There are situations, however, in which semantic intuition leaves us in the darkness, especially when dealing with complex sentences involving different levels of modal truths, such as the following:

\vspace{-0.5cm}
\begin{align}
	\tag{4} \mbox{If 1 plus 1 is equal to 2, then it's raining in Moscow at 0h01 on January 1, 2032.}
\end{align}

\noi It is very difficult to say whether, according to our linguistic intuitions, (4) should receive the value necessarily true or contingently true in, namely, 2022.

Anyway, we will explore here just one of Ivlev's possible interpretations for modal implication by four values. First, because the reader will be able to check that this Nmatrix is semantically intuitive, as we already argued in \cite{con:cer:per:19}. In addition, some Ivlev implication tables are too strong from the point of view of relational semantics, that is, it ends up making certain propositional formulas valid that are not valid even in the strongest normal modal system in relational semantics, which is {\bf S5}. Finally, the reader will be able to check that the non-deterministic implication below proposed by Ivlev  coincides with the one proposed independently by Kearns in the eighties in \cite{kear:81}. For these reasons, from now on we will adopt the following Nmatrix for the operator $\tilde{\to}$ :

\begin{displaymath}
	\def\arraystretch{1.3}
	\begin{array}{|l|l|l|l|l|}
		\hline \tilde{\to}  & \laT 	& \lat 			& \laf 			& \laF \\
		\hline \hline \laT 	& \laT 	& \lat 			& \laf			& \laF \\
		\hline \lat			& \laT 	& \{\laT,\lat\} & \laf 			& \laf \\
		\hline \laf			& \laT 	& \{\laT,\lat\} & \{\laT,\lat\} & \lat \\
		\hline \laF			& \laT	& \laT 			& \laT 			& \laT \\	
		\hline
	\end{array}
\end{displaymath}

\

\noindent Finally, the  multioperator assigned to $\Box$ must capture the notion of \emph{necessary} in natural language. Consider, for instance, the sentence:

\vspace{-0.5cm}
\begin{align}
	\tag{5} \mbox{Socrates is mortal.}
\end{align}

\noi Suppose we consider, in some sense, sentence (5) to be necessarily true. Thus, we would infer that the sentence

\vspace{-0.5cm}
\begin{align}
	\tag{6} \mbox{Socrates is necessarilly mortal.}
\end{align}

\noi is true. But if (5) is necessarily true, should (6) be a sentence necessarily true or only contingently true?
Reciprocally, if (5) is contingently true, then (6) is false, but should (6) be only contingently false or impossible? 

These modal puzzles seem to be a consequence of the fact that iterated modalities are very rare in natural language. This seems to be one of the causes of the enormous quantity of propositional modal systems that exist in the literature.

Ivlev was aware of this fact, so he presented more than one table to interpret the operador $\Box$. Here we are going to work with the following tables:\footnote{$\Box_1$ and $\Box_3$ were proposed in~\cite{ivl:88}, while $\Box_2$ was proposed in~\cite{con:cer:per:15}. The intuition behind $\Box_2$ becomes clearer in the context of {\em swap structures}, see~\cite{con:gol:19}.}

\begin{displaymath}
	\def\arraystretch{1.3}
	\begin{array}{|l|l|l|l|}
		\hline  	   			  	 & \tilde{\Box}_1 	& \tilde{\Box}_2 & \tilde{\Box}_3 \\
		\hline \hline  			\laT & \{\laT,\lat\} 	& \laT 			 & \laT \\
		\hline 					\lat & \{\laf,\laF\}  	& \{\laf,\laF\}	 & \laF \\
		\hline 					\laf & \{\laf,\laF\}  	& \{\laf,\laF\}  & \laF \\
		\hline 					\laF & \{\laf,\laF\} 	& \{\laf,\laF\}  & \laF \\
		\hline	
	\end{array}
\end{displaymath}

\
\noi Taking these operators into account, we can define three distinct Ivlev-like modal logics with a corresponding four-valued Nmatrix semantics:

\begin{itemize}
	\item $\mathcal{M}({\bf Tm}) = \langle \{\laT,\lat,\laf,\laF \}, \{\laT,\lat\}, \{\tilde{\neg},\tilde{\to}, \tilde{\Box}_1 \} \rangle$
	\item $\mathcal{M}({\bf S4m}) = \langle \{\laT,\lat,\laf,\laF \}, \{\laT,\lat\}, \{\tilde{\neg},\tilde{\to}, \tilde{\Box}_2 \} \rangle$
	\item $\mathcal{M}({\bf S5m}) = \langle \{\laT,\lat,\laf,\laF \}, \{\laT,\lat\}, \{\tilde{\neg},\tilde{\to}, \tilde{\Box}_3 \} \rangle$
	
\end{itemize}

\noi We have analyzed these systems in~\cite{con:cer:per:15}, \cite{con:cer:per:17} and \linebreak \cite{con:cer:per:19} (see also~\cite{omo:sku:16} and~\cite{paw:larosa:21}).


\section{Extensions to quantified languages} \label{secTm*}

In~\cite{con:cer:per:21} the extension of the modal systems {\bf Tm}, {\bf S4m} and {\bf S5m} to first-order languages was analyzed. We briefly recall the main definitions and basic results obtained therein.

Let us begin by using our semantic intuitions to understand a sentence quantified in natural language, like the sentence below:

\vspace{-0.5cm}
\begin{align}
	\tag{7} \mbox{Everybody is mortal.}
\end{align}

\noi Sentence (7) will  be necessarily true when it is necessarily true for each individual in the domain. But sentence (7) will only be contingently true if: (i) there is at least one individual in the domain who is contingently mortal; and (ii) every individual in the domain is mortal, necessarily or only contingently. We say that (7) is contingently false if at least one individual in the domain is not mortal. Furthermore, any individual in the domain could be mortal, that is, it is not impossible for any individual to be mortal. Finally, (7) is impossible when it is not possible for at least one individual to be mortal.

Keeping these intuitions in mind, let us briefly recall the semantics of first-order structures for  {\bf Tm}$^\approx$ introduced in \cite{con:cer:per:21}. By simplicity, and given that in this paper we are mainly interested in tableau systems for some quantified Ivlev-like modal logics, the equality predicate $\approx$ will not be considered, and the signatures will not include symbols for functions. From now on, we will call  {\bf Tm}$^*$ the first-order extension of {\bf Tm} without the equality predicate or any symbol for functions.

Formally, a {\em (basic) predicate signature} is a collection $\Theta$ formed by the following symbols:  (i) a non-empty set of predicate symbols $\mathcal{P}$, with the corresponding arity $\varrho(P) \geq 1$ for each $P \in \mathcal{P}$; (ii) a possible empty set of  individual constants $\mathcal{C}$. It will also assumed a fixed infinite denumerable set $Var =\{x_1,x_2,\ldots\}$ of individual variables.\footnote{It should be noted that most part of modal logic manuals --- for instance \cite{hug:cre:96}, \cite{fit:med:98} and \cite{gar:06} --- only consider basic predicate signatures, that is, do not consider function symbols among the symbols of their language (an exception is \cite[p. 241]{car:piz:08}).}  

A term $\tau$ in a predicate language $\Theta$ is a variable or a constant. Given a predicate signature $\Theta$, the set $For(\Theta)$ of well-formed formulas (wffs) is also defined recursively as follows: (i) for each $n$-ary predicate $P$, if $\tau_1, \ldots, \tau_n$ are terms, then $P\tau_1\ldots \tau_n$ is a wff (called atomic); (ii) if $\varphi$ is a wff and $x$ is a variable, then $(\neg \varphi)$, $(\Box \varphi)$ and $(\forall x \varphi)$ are also wffs; (iii) if $\varphi$ and $\psi$ are wffs, then $(\varphi \to \psi)$ is also a wff; (iv) nothing else is a wff. As before,  parenthesis will be omitted when readability is unambiguous.

Recall from \cite{con:cer:per:21} that quantifiers are interpreted in \tms\ by means of the following (deterministic) multioperators $\tilde{Q}^d_4: (\mathcal{P}(\{\laT,\lat,\laf,\laF\}) \setminus \{\emptyset\}) \to (\mathcal{P}(\{\laT,\lat,\laf,\laF\}) \setminus \{\emptyset\})$, for  $Q\in \{\forall, \exists\}$:

\begin{displaymath}
	\begin{array}{|c|c|}
		\hline  X & \tilde{\forall}_4^d(X)\\
		\hline \{\laT\} & \laT  \\
		\hline \{\lat\} & \lat  \\
		\hline \{\laT, \ \lat\} & \lat \\
		\hline \{\laf, \ \lat\} & \laf \\
		\hline \{\laf, \ \lat, \ \laT\} & \laf \\
		\hline \{\laf\} & \laf \\
		\hline \{\laf, \ \laT\} & \laf \\
		\hline \laF \in X & \laF \\
		\hline
	\end{array}
	\hspace{1cm}
	\begin{array}{|c|c|}
		\hline  X & \tilde{\exists}_4^d(X)\\
		\hline \laT \in X & \laT  \\
		\hline \{\lat\} & \lat  \\
		\hline \{\lat, \ \laF\} & \lat \\
		\hline \{\lat, \laf\} & \lat \\
		\hline \{\lat, \ \laf, \laF\} & \lat \\
		\hline \{\laf\} & \laf \\
		\hline \{\laf, \ \laF\} & \laf \\
		\hline \{\laF\} & \laF \\
		\hline
	\end{array}
\end{displaymath}

\

\noi Such quantifiers are deterministic by definition, and correspond, respectively, to the deterministic conjunction and disjunction of the members of $X$ according to the order given by the chain $\laF \leq \laf \leq \lat \leq \laT$. As it was done in \cite{con:cer:per:21}, by simplicity only the universal quantifier will be considered in \tms, and $\exists x \varphi$ will be an abbreviation for $\neg\forall x \neg\varphi$.

\begin{defi} \label{structure}
	Let $\Theta$ be a predicate signature. A four-valued modal structure over $\Theta$ is a pair $\mathfrak{A}= \langle U, \cdot^\mathfrak{A} \rangle$, such that $U$ is a non-empty set (the {\em domain} of the structure) and $\cdot^\mathfrak{A}$ is an interpretation function for the symbols of $\Theta$, which is defined as follows:
	\begin{itemize}
		\item For each $n$-ary predicate $P$, $P^\mathfrak{A}:U^n \to \{\laT,\lat,\laf,\laF\}$ is a function;
		\item For each individual constant $c$, $c^\mathfrak{A}$ is an element of $U$.
	\end{itemize}
\end{defi}

\begin{defi} \label{diagram} Let $\mathfrak{A}= \langle U, \cdot^\mathfrak{A} \rangle$ be a four-valued modal structure over a signature $\Theta$ as in Definition~\ref{structure}, and let $C^\mathfrak{A} = \{c^\mathfrak{A} \ : \ c \in C\}$. Let $\bar U= \{ \bar a \ : \ a  \in U\setminus C^\mathfrak{A}\}$ be a set of new constant symbols (i.e., disjoint from $C$), and let $\Theta_U$ be the signature obtained from $\Theta$ by adding the set $\bar U$ of constants. Let $\mathfrak{A}_U= \langle U, \cdot^{\mathfrak{A}_U} \rangle$ be the expansion of $\mathfrak{A}$ to $\Theta_U$ by setting that $\bar{a}^{\mathfrak{A}_U}=a$ for every $\bar a \in \bar U$.  If $C^\mathfrak{A} = U$ then, by definition, $\Theta_U=\Theta$ and  $\mathfrak{A}_U= \mathfrak{A}$.
\end{defi}

\begin{obs} \label{diag-cons-rem} If  $\mathfrak{A}$  is a four-valued modal structure over $\Theta$ and $\mathfrak{A}_U$ is defined as above, both structures should validate the same closed formulas over $\Theta$. This will be guaranteed by using valuations over $Sen(\Theta_U)$, to be defined below. Observe that  $C \cup \bar U$ is the set of constants of $\Theta_U$ and, for every $a \in U$, there is a constant $c$ in $\Theta_U$ such that $c^{\mathfrak{A}_U}=a$. That is, $(C \cup \bar{U})^{\mathfrak{A}_U}=U$.
\end{obs}

\begin{defi} [da Costa] \label{variant} Let $\varphi$ and $\psi$ be formulas. If $\varphi$ can be obtained from $\psi$ by means of addition or deletion of void quantifiers,\footnote{That is, a quantifier $\forall x \varphi$ or $\exists x \varphi$ such that $x$ does not occur free in $\varphi$ (recalling that $\exists x \varphi$ stands for $\neg\forall x \neg \varphi$).} or by renaming bound variables (keeping the same free variables in the same places), we say that $\varphi$ and $\psi$ are {\em variant} of each other, and it will denoted by $\varphi \sim \psi$.
\end{defi}

\noindent
From now on, we will write $\varphi[x/\tau]$ to denote the formula obtained from $\varphi$ by replacing simultaneously every free occurrence of the variable $x$ by the term $\tau$, provided that $\tau$ is free for $x$ in $\varphi$.\footnote{Recall that a term $\tau$ is free for a variable $x$ in a formula $\varphi$ if the following holds: if a free occurrence of $x$ in $\varphi$ lies in the scope of a quantifier $\forall y$, then $y$ does not occur in $\tau$.} Note that, in the present framework, $\tau$ is either an individual variable or a constant symbol. If $\tau$ is a constant, then $\tau$ is always free for $x$ in any formula. If $\tau$ is a variable $z$ then $\tau$ is free for $x$ in $\varphi$ if the following holds: if a free occurrence of $x$ in $\varphi$ lies in the scope of a quantifier $\forall y$, then $y\neq z$.

\begin{defi} \label{Tm-sem}
	Let $\mathfrak{A}$ and $\mathfrak{A}_U$ be a four-valued modal structure as in Definition~\ref{diagram}. A \tms-{\em valuation} over $\mathfrak{A}$ is a function $v:Sen(\Theta_U) \to \{\laT,\lat,\laf,\laF\}$ defined recursively as follows:\footnote{The notion of valuations over a Nmatrix and a first-order structure considered here is slightly different of the one considered in~\cite{con:cer:per:21}. Specifically, we will not require the satisfaction of the substitution lemma, see Remark~\ref{lemsubs} below.}
	\begin{enumerate}
		\item For atomic formulas of the form $P c_1\ldots c_n$, $v(P c_1\ldots c_n) = P^\mathfrak{A}(c_1^{\mathfrak{A}_U},\ldots, c_n^{\mathfrak{A}_U})$;
		
		\item $v(\neg \varphi) \in \tilde{\neg} \, v(\varphi)$;
		
		\item $v(\Box \varphi) \in \tilde{\Box}_1 \, v(\varphi)$;
		
		\item $v(\varphi \to \psi) \in v(\varphi) \,\tilde{\to}\, v(\psi)$;
		
		\item For formulas of the form  $\forall x  \varphi$, consider the set $X(\varphi,x,v) = \big\{v(\varphi[x/c]) \ : \  c \in C \cup \bar U\big\}$. Then, $v(\forall x \varphi) \in \tilde{\forall}_4^d\big(X(\varphi,x,v)\big)$, where  $\tilde{\forall}_4^d$ is defined as above.
		
		\item If $\varphi \sim \varphi'$ then $v(\varphi) = v(\varphi')$.
	\end{enumerate}
\end{defi}

\noi With a slight change in the definition above, we can define valuations for the other modal systems studied here:

\begin{defi} \label{S4m-sem}
	Let $\mathfrak{A}$ and $\mathfrak{A}_U$ be  a four-valued modal structure as in Definition~\ref{diagram}. A \sqms-{\em valuation} over $\mathfrak{A}$ is a function $v:Sen(\Theta_U) \to \{\laT,\lat,\laf,\laF\}$ defined recursively exactly as in Definition \ref{Tm-sem}, with a single change in clause 2:  
	
	\begin{enumerate}
		\setcounter{enumi}{1}
		\item $v(\Box \varphi) \in \tilde{\Box}_2 \, v(\varphi)$;		
	\end{enumerate}
\end{defi}

\begin{defi} \label{S5m-sem}
	Let $\mathfrak{A}$ and $\mathfrak{A}_U$ be  a four-valued modal structure as in Definition~\ref{diagram}. A  \scms-{\em valuation} over $\mathfrak{A}$ is a function $v:Sen(\Theta_U) \to \{\laT,\lat,\laf,\laF\}$ defined recursively exactly as in Definition \ref{Tm-sem}, with a single change in clause 2:  
	\begin{enumerate}
		\setcounter{enumi}{1}
		\item $v(\Box \varphi) \in \tilde{\Box}_3 \, v(\varphi)$;
		
	\end{enumerate}
\end{defi}

\noi From now on, we will use {\bf Lm} (\lms, resp.) to indistinctly denote \tm, \sqm\ or \scm\ (\tms, \sqms\ or \scms, resp.).

\begin{defi} \label{semconsFOL} Let $\Gamma \cup \{\varphi\} \subseteq For(\Theta)$ such that $Var(\Gamma \cup \{\varphi\}) \subseteq \{x_1,\ldots, x_n\}$.  Then, $\varphi$ is a semantic consequence of $\Gamma$ in a quantified modal logic \lms, denoted by $\Gamma \models_{\bf Lm^*} \varphi$, if, for every four-valued modal structure $\mathfrak{A}$ over  $\Theta$ and for a \lms-valuation $v$ over $\mathfrak{A}$, if $v(\gamma[x_1/c_1 \cdots x_n/c_n]) \in \{\laT,\lat\}$ for every $\gamma \in \Gamma$  and every $c_1,\ldots, c_n \in C \cup \bar U$ then  $v(\varphi[x_1/c_1 \cdots x_n/c_n]) \in \{\laT,\lat\}$ for every $c_1,\ldots, c_n \in C \cup \bar U$.
\end{defi}


\section{Hilbert calculi} \label{secAxs}

In this section, we present the Hilbert calculi for the modal logics to be studied along this paper. It should be observed that {\bf Tm} and {\bf S5m} were introduced in~\cite{ivl:88} under the names of {\bf Sa}$^+$ and {\bf Sb}$^+$, respectively.\footnote{As observed in~\cite{omo:sku:16}, the inference rules considered by Ivlev concerning the replacement of $\varphi$ by $\neg\neg\varphi$ inside any formula are not sound, and they must be changed by the axioms 	\axDNu\ and \axDNd\ below.}

Let consider the following axiom schemas and inference rules:\\

{\bf Axiom schemas:}\\

$\begin{array}{ll}
	\axu & \varphi \to (\psi \to \varphi)\\[2mm]
	\axd & (\varphi \to (\psi \to \xi)) \to ((\varphi \to \psi) \to (\varphi \to \xi))\\[2mm]
	\axt & (\neg \psi \to \neg \varphi) \to ((\neg \psi \to \varphi) \to \psi)\\[2mm]
	\axq & \forall x \varphi \to \varphi[x/\tau] \  \ \mbox{ if $\tau$ is free for $x$ in $\varphi$}\\[2mm]
	\axc & \forall x (\varphi \to \psi) \to (\varphi \to \forall x\psi) \ \  \ \mbox{ if $\varphi$ contains no free  occurrences of $x$}\\[2mm]
	\axs & \varphi\to\psi \ \ \mbox{ if $\varphi \sim \psi$}\\[2mm]
	\axK & \Box (\varphi \to \psi) \to (\Box \varphi \to \Box \psi)\\[2mm]
	\axKu & \Box (\varphi \to \psi) \to (\Box \neg\psi \to \Box \neg \varphi))\\[2mm]
	\axKd & \neg \Box \neg(\varphi \to \psi) \to (\Box \varphi \to \neg \Box \neg \psi)\\[2mm]
	\axMu & \Box \neg \varphi \to \Box(\varphi \to \psi)\\[2mm]
	\axMd & \Box \psi \to \Box(\varphi \to \psi)\\[2mm]
	\axMt & \neg \Box \neg \psi \to \neg \Box \neg (\varphi \to \psi) \\[2mm]
	\axMc & \neg \Box \neg\neg \varphi \to \neg \Box \neg (\varphi \to \psi) \\[2mm]
	\axT & \Box \varphi \to \varphi\\[2mm]
	\axSq & \neg \Box \neg \Box \varphi \to\Box \varphi\\[2mm]
	\axSc & \Box \varphi \to \Box \Box \varphi
		\end{array}
		$
		
		$\begin{array}{ll}
	\axDNu & \Box \varphi \to \Box \neg \neg \varphi\\[2mm]
	\axDNd & \Box \neg \neg \varphi \to \Box \varphi\\[2mm]
	\axBF & \forall x \Box \varphi \to \Box \forall x \varphi\\[2mm]
	\axCBF & \Box \forall x \varphi \to \forall x \Box \varphi\\[2mm]
	\axNBF & \forall x \neg \Box \neg  \varphi \to \neg \Box \neg \forall x \varphi \\[2mm]
	\axPBF &  \neg \Box \neg \forall x \varphi \to \forall x \neg \Box \neg  \varphi\\[6mm]
	
	\mbox{{\bf Inference rules:}}\\[2mm]
	
	\MP: &  \psi \ \  \ \mbox{ follows from $\varphi$ and $\varphi \to \psi$}\\[2mm]
	\Gen: &  \forall x \varphi \ \  \ \mbox{ follows from $\varphi$}\\[4mm]
\end{array}
$

\

\noi Taking into account the above axioms and rules, we can consider the following systems:

\begin{itemize}
	\item ${\bf CL} = \{\axu,\axd, \axt, \MP \}$
	\item {\bf CL*}$ = {\bf CL} \cup \{\axq,\axc, \Gen \}$
	\item ${\bf Tm} = {\bf CL} \cup \{\axK, \axKu, \axKd
	\axMu,\axMd, \axMt, \axMc, \axT,  \axDNu, \axDNd \}$
	\item \tms$ = {\bf Tm} \cup \{\axq,\axc, \axBF, \axCBF, \axNBF, \axPBF, \Gen \}$
	\item ${\bf S4m} ={\bf Tm} \cup \{ \axSq\}$
	\item \sqms = \tms$ \cup \{\axSq \}$
	\item ${\bf S5m} ={\bf S4m} \cup \{ \axSc\}$	
	\item \scms = \sqms$ \cup \{\axSc \}$
\end{itemize}

\noi The notion of derivation in a logic {\bf L}  is defined as usual.  We will use the conventional notation $\Gamma \vdash_{\bf L} \varphi$ in order to express that there is a derivation in  {\bf L} of $\varphi$ from $\Gamma$.

Any logic {\bf Lm} satisfies the Deduction metatheorem (DMT):

\begin{teo}[Deduction Metatheorem (DMT) for  {\bf Lm}]\label{DM} \
	Suppose that there exists in {\bf Lm} a derivation of $\psi$ from $\Gamma \cup\{\varphi\}$. Then $\Gamma \vdash_{\bf Lm} \varphi \to \psi$.\footnote{A detailed version of this proof for {\bf CL}, which also holds for {\bf Lm}, can be found in~\cite[Proposition~1.9]{men:10}.}
\end{teo}

\noi  As it could be expected, given that no inference rule was added to {\bf CL*} to obtain \lms, each \lms\ satisfies the restricted version of the Deduction metatheorem (DMT), as usually presented in {\bf CL*}:

\begin{teo}[Deduction Metatheorem (DMT) for  {\lms}] \label{DMQ}
	Suppose that there exists in {\lms} a derivation of $\psi$ from $\Gamma \cup\{\varphi\}$, such that no application of the rule (Gen) has, as its quantified variable, a free variable of $\varphi$ (in particular, this holds when $\varphi$ is a sentence). Then $\Gamma \vdash_{\lms} \varphi \to \psi$.\footnote{A detailed version of this proof for {\bf CL*}, which also holds for \lms, can be found in~\cite[Proposition~2.5 and Corollaries~2.6 and~2.7]{men:10}.}
\end{teo}

\begin{teo}[Soundness and Completeness for {\bf Lm}]
	Let $\Gamma \cup\{\alpha\} \subseteq For_P$ be a set of  formulas. Then: $\varphi$ is a 
	semantic consequence of $\Gamma$ over the  Nmatrix $\mathcal{M}({\bf Lm})$ iff $\Gamma \vdash_{\bf Lm} \varphi$.
\end{teo}

\noi A detailed proof of this result can be found in~\cite{con:cer:per:15,con:cer:per:17}.


\begin{obs} \label{lemsubs} The \tms-valuations considered in~\cite{con:cer:per:21} require, in addition to the clauses in Definition~\ref{Tm-sem}, the satisfaction of the Leibniz rule for the equality predicate $\approx$, as well as the satisfaction of the substitution lemma, namely: $v(\varphi[x/\tau],s)=v(\varphi,s^x_a)$ where $s$  is any assignment for variables (that is, a function $s:Var \to U$, where $U$ is the domain of the given first-order structure),  $a$ is the value assigned to the term $\tau$ in the given first-order structure by using $s$, and $s^x_a$ is the assignment obtained from $s$ by assigning the value $a$ to $x$.
It is easy to translate the semantical framework of~\cite{con:cer:per:21}  to the present one: if $\varphi=\varphi(x_1,\ldots,x_n,x)$ is a formula having  (at most) the variables $x_1,\ldots,x_n,x$ occurring free, $s$ is an assignment for variables and $a \in U$ then $v(\varphi,s)$ and $v(\varphi,s^x_a)$ correspond in the present setting  to $v(\varphi[x_1/c_1\cdots x_n/c_n \ x/c])$ and $v(\varphi[x_1/c_1\cdots x_n/c_n \ x/\bar a])$, respectively (here, $c_1,\ldots,c_n,c \in C \cup \bar U$ and $\bar a \in \bar U$). 
Hence, the semantical framework in~\cite{con:cer:per:21} can be translated to the present one, but taking into consideration that, in the former,  the valuations satisfy the substitution lemma and the Leibniz rule. This produces subtle differences between both approaches:  if $c_1 \neq c_2$ in $C$ are such that $c_1^{\mathfrak{A}}=a=c_2^{\mathfrak{A}}$ then, according to Definition~\ref{Tm-sem}, the values $v(\varphi[x/c_1])$ and   $v(\varphi[x/c_2])$ are allowed to be different. On the other hand, in the framework considered in~\cite{con:cer:per:21} we have in this case, by the substitution lemma, that $v(\varphi[x/c_1],s)=v(\varphi,s^x_a)=v(\varphi[x/c_2],s)$, for every valuation $v$ and every assignment $s$.  Despite these small technical differences, both semantical consequence relations coincide, characterizing \tms\ (without the equality predicate $\approx$, as we shall see in Theorem~\ref{sound-compl-lms} below). It should be observed that the changes done in the present semantical framework w.r.t. the one considered in~\cite{con:cer:per:21}  simplify the definition of the tableau systems, as well as the corresponding proofs of soundness and completeness to be presented in the next sections of the paper.  
\end{obs}

\begin{teo}[Soundness and Completeness for  \lms] \label{sound-compl-lms}
	Let $\Gamma \cup\{\alpha\} \subseteq For(\Theta)$ be a set of  formulas. Then: $\Gamma \models_{\lms}\varphi$ iff $\Gamma \vdash_{\lms} \varphi$.\footnote{Recall that a  proof of this result for the case of \tms\ with identity predicate $\approx$ can be found in~\cite{con:cer:per:21}. That proof can be easily adapted to \sqms\ and \scms. However, such results concern the semantical framework defined therein which, as observed in Remark~\ref{lemsubs}, differs slightly from the present setting.}
\end{teo}

\begin{proof} 	\ \\
({\bf Soundness}) It is easy to see that the notion of valuation considered here is sufficient to guarantee the soundness of the axioms and inference rules of \lms, taking into account that no function symbols are allowed in the signatures. In special, it validates axiom \axq\ (in which the substitution lemma plays a fundamental role in~\cite{con:cer:per:21}). Thus, let $\psi = \forall x \varphi \to \varphi[x/\tau]$ be an instance of  axiom \axq\ over $\Theta$ (hence $\tau$ is a term free for $x$ in $\varphi$). Let  $\mathfrak{A}$ be a four-valued modal structure over $\Theta$ with domain $U$, and let $v$ be a  \lms-valuation over $\mathfrak{A}$. Let $\vec{x}=x_1\ldots x_n$ be a finite sequence of distinct variables such that $Var(\psi) \subseteq \{x_1,\ldots,x_n\}$ and let $\vec{c}=c_1\ldots c_n$ be a finite sequence of constants in $C \cup \bar{U}$. We want to prove that $v(\psi[\vec{x}/\vec{c}]) \in \{\laT,\lat\}$.  If $x$ does not occur free in $\varphi$ then the result  is clearly true. Indeed, in such case,  $v(\forall x\varphi[\vec{x}/\vec{c}]) = v(\varphi[\vec{x}/\vec{c}])=v(\varphi[x/\tau][\vec{x}/\vec{c}])$. Now, suppose that $x$ occurs free in $\varphi$. If $x=x_i$ for some $1 \leq i \leq n$ let $\vec{x'}=x_1\ldots x_{i-1}x_{i+1}\ldots x_n$ and $\vec{c'}=c_1\ldots c_{i-1}c_{i+1}\ldots c_n$. Otherwise, let $\vec{x'}=\vec{x}$ and  $\vec{c'}=\vec{c}$. With this notation, it is easy to see that $\forall x \varphi[\vec{x}/\vec{c}]= \forall x (\varphi[\vec{x'}/\vec{c'}])$. Hence, $v(\forall x \varphi[\vec{x}/\vec{c}])= v(\forall x (\varphi[\vec{x'}/\vec{c'}]))$. If $v(\forall x \varphi[\vec{x}/\vec{c}]) \in \{\laf,\laF\}$ then, by Definition of $\tilde{\to}$, $v(\psi[\vec{x}/\vec{c}]) \in \{\laT,\lat\}$. Suppose now that $v(\forall x \varphi[\vec{x}/\vec{c}])= v(\forall x (\varphi[\vec{x'}/\vec{c'}])) \in \{\laT,\lat\}$. We want to prove that $v(\varphi[x/\tau][\vec{x}/\vec{c}]) \in \{\laT,\lat\}$. Let $X=\{v(\varphi[\vec{x'}/\vec{c'}][x/c]) \ : \ c \in C \cup \bar{U}\}$.
By Definition~\ref{Tm-sem}(5), $v(\forall x \varphi[\vec{x}/\vec{c}]) \in \tilde{\forall}_4^d(X)$, whence $\tilde{\forall}_4^d(X) \subseteq  \{\laT,\lat\}$. Thus, by definition of  $\tilde{\forall}_4^d$, $X \subseteq  \{\laT,\lat\}$. That is,
$$(*) \hspace{1cm}v(\varphi[\vec{x'}/\vec{c'}][x/c]) \in \{\laT,\lat\} \ \mbox{ for every $c \in C \cup \bar{U}.$}$$
We have two cases to analyze: \\[1mm]
(1) $\tau$ is a variable $y$ free for $x$ in $\varphi$. Then, $y=x_j$ for some $1 \leq j \leq n$, given that  $Var(\varphi[x/\tau]) \subseteq \{x_1,\ldots,x_n\}$; or\\[1mm]
(2) $\tau$ is a constant $c_0 \in C$. In both cases 
$v(\varphi[x/\tau][\vec{x}/\vec{c}])= v(\varphi[\vec{x'}/\vec{c'}][x/c_k]) \in X$, 
where $k=j$ (in case (1)) or $k=0$ (in case (2)). By $(*)$,  $v(\varphi[x/\tau][\vec{x}/\vec{c}]) \in \{\laT,\lat\}$. This shows that $v(\psi[\vec{x}/\vec{c}]) \in \{\laT,\lat\}$ as required.

The validity of the other axioms can be proved by an easy adaptation (and simplification) of the proof of soundness of \tms\ given in~\cite[Subsection~2.4]{con:cer:per:21}. 	The reader can check the details.\\[2mm]	
({\bf Completeness)} The proof for \lms\ by using for the structures and valuations considered here can be easily adapted from the one obtained in~\cite{con:cer:per:21} as follows (in order to fix ideas, only  the case of \tms\ will be considered). Recall first the following notions and results: let  {\bf L} be a  Tarskian and finitary logic defined over a set of formulas $For$, and let $\varphi \in For$. A set of formulas $\Delta \subseteq For$  is {\em $\varphi$-saturated in {\bf L}} if $\Delta \nvdash_{\bf L} \varphi$ but $\Delta,\psi\vdash_{\bf L}\varphi$ for every $\psi \in For\setminus\Delta$. By a well-known result by Lindenbaum and \L o\'s (see~\cite[Theorem~22.2]{woj:84}), if $\Gamma \nvdash_{\bf L} \varphi$ then there exists a $\varphi$-saturated set $\Delta$  in {\bf L} such that $\Gamma \subseteq \Delta$, whenever {\bf L} is Tarskian and finitary. In particular,   we have:\\[1mm]
{\bf Fact 1:} Let  $\Gamma \cup \{\varphi\} \subseteq For(\Theta)$ such that $\Gamma \nvdash_{\tms} \varphi$. Then, there exists a set of formulas $\Delta$ such that $\Gamma \subseteq \Delta$ and $\Delta$ is $\varphi$-saturated in \tms.\\[1mm]
It is easy to prove that a $\varphi$-saturated set $\Delta$ in \tms\ is a closed theory (that is: $\psi \in \Delta$ iff $\Delta \vdash_{\tms} \psi$) and the following holds: $\psi \in \Delta$ iff $\neg\psi \notin \Delta$, and $\psi \to \gamma \in \Delta$ iff either $\psi \notin \Delta$ or $\gamma \in \Delta$.

Given a set of formulas $\Gamma \subseteq For(\Theta)$ and a set $C_0 \subseteq C$ of constants, $\Gamma$ is said to be a {\em $C_0$-Henkin theory in} \tms\ if, for every formula $\psi$ with at most a free variable $x$, there exists a constant $c \in C_0$ such that $\Gamma \vdash_{\tms} \psi[x/c] \to \forall x \psi$. Let $\Theta_{C'}$ be the signature obtained from $\Theta$ by adding a set $C'$  of new constants, and let $\vdash_{\tms}^{C'}$ be the corresponding consequence relation of \tms\ over $\Theta_{C'}$. By a standard argument it can proved the following:\\[1mm]
{\bf Fact 2:} Every  $\Gamma \subseteq For(\Theta)$ can be conservatively extended to a  $C'$-Henkin theory $\Gamma' \subseteq For(\Theta_{C'})$ in \tms. That is: $\Gamma \subseteq \Gamma'$, $\Gamma'$ is a  $C'$-Henkin theory in \tms\ over  $\Theta_{C'}$, and $\Gamma \vdash_{\tms} \varphi$ iff  $\Gamma' \vdash_{\tms}^{C'} \varphi$ for every $\varphi \in For(\Theta)$. Moreover, if $\Gamma' \subseteq \Gamma'' \subseteq  For(\Theta_{C'})$ then $\Gamma''$ is also a $C'$-Henkin theory in \tms.

Now, let   $\Gamma \cup \{\varphi\} \subseteq For(\Theta)$ such that $\Gamma \nvdash_{\tms} \varphi$. We will prove that $\Gamma \not\models_{\tms}\varphi$. In order to do this, let us observe first that,by {\bf Fact 2}, there exists a $C'$-Henkin theory $\Gamma'$ in \tms\ over  $\Theta_{C'}$ for a new set of constant symbols $C'$ such that $\Gamma'$ extends conservatively $\Gamma$. From this, $\Gamma' \nvdash_{\tms}^{C'} \varphi$ and so, by {\bf Fact 1}, there exists a $\varphi$-saturated theory $\Delta$ in \tms\ over $\Theta_{C'}$ extending $\Gamma'$. By the last part of {\bf Fact 2}, $\Delta$ is also a $C'$-Henkin theory over $\Theta_{C'}$ in \tms.

The {\em canonical} four-valued modal structure $\mathfrak{A}_\Delta= \langle C \cup C', \cdot^\mathfrak{A}_\Delta \rangle$ over $\Theta_{C'}$ is defined as follows: $c^{\mathfrak{A}_\Delta}=c$ for every constant symbol $c \in C \cup C'$ and, for every $n$-ary predicate symbol $P$, the function $P^{\mathfrak{A}_\Delta}$ is defined as follows:

$$P^{\mathfrak{A}_\Delta}(c_1,\ldots,c_n)=\left\{
                    \begin{array}{ll}
                      \laT, & \hbox{if $P(c_1,\ldots,c_n) \in \Delta$ and $\Box P(c_1,\ldots,c_n) \in \Delta$;} \\[2mm]
                      \lat, & \hbox{if $P(c_1,\ldots,c_n) \in \Delta$ and $\neg\Box P(c_1,\ldots,c_n) \in \Delta$;} \\[2mm]
                       \laf, & \hbox{if $\neg P(c_1,\ldots,c_n) \in \Delta$ and $\neg\Box \neg P(c_1,\ldots,c_n) \in \Delta$;} \\[2mm]
                        \laF, & \hbox{if $\neg P(c_1,\ldots,c_n) \in \Delta$ and $\Box \neg P(c_1,\ldots,c_n) \in \Delta$.}
                    \end{array}
                  \right.
$$
Observe that, for every $\psi$, either $\psi \in \Delta$ or $\neg\psi \in \Delta$ (but not both simultaneously). This shows that $P^{\mathfrak{A}_\Delta}$ is well-defined. Let $\mathfrak{A}$ be the reduct of $\mathfrak{A}_\Delta$ to $\Theta$. Then  $c^\mathfrak{A}=c$ for every constant symbol $c \in C$ and $P^{\mathfrak{A}}=P^{\mathfrak{A}_\Delta}$ for every predicate symbol $P$. Moreover, since  $U=C \cup C'$ is the domain of $\mathfrak{A}$ and $c^{\mathfrak{A}_\Delta}=c$ for every $c \in C \cup C'$  then $\bar{U}$, $\Theta_U$ and $\mathfrak{A}_U$ as in Definition~\ref{diagram} will be identified, respectively, with $C'$, $\Theta_{C'}$ and $\mathfrak{A}_\Delta$.
The {\em canonical} valuation over $\mathfrak{A}$ is the function $v_\Delta:Sen(\Theta_{C'}) \to \{\laT,\lat,\laf,\laF\}$ defined as follows:

$$v_\Delta(\psi)=\left\{
                    \begin{array}{ll}
                      \laT, & \hbox{if $\psi \in \Delta$ and $\Box \psi \in \Delta$;} \\[2mm]
                      \lat, & \hbox{if $\psi \in \Delta$ and $\neg\Box \psi \in \Delta$;} \\[2mm]
                       \laf, & \hbox{if $\neg \psi \in \Delta$ and $\neg\Box \neg \psi \in \Delta$;} \\[2mm]
                        \laF, & \hbox{if $\neg \psi \in \Delta$ and $\Box \neg \psi \in \Delta$.}
                    \end{array}
                  \right.
$$
Then, $v_\Delta$ is a valuation over $\mathfrak{A}$ (or, equivalently, over $\mathfrak{A}_\Delta$).
The proof of this fact is similar, but  simpler, than the one given in~\cite[Lemma~2.27]{con:cer:per:21}. By the very definition, $v_\Delta(\psi) \in \{\laT,\lat\}$ iff $\psi \in \Delta$. Let $\psi \in \Gamma$, and let $x_1,\ldots,x_n$ be a list of variables containing all the variables occurring free in $\psi$. Given $c_1,\ldots,c_n \in C \cup C'$ we infer that $\psi[\vec{x}/\vec{c}] \in \Delta$, by combining  (\Gen) and \axq\ and by the fact that $\Delta$ is a closed theory containing $\Gamma$. That is,  $v_\Delta(\psi[\vec{x}/\vec{c}]) \in \{\laT,\lat\}$ for every $\psi \in \Gamma$ and every $c_1,\ldots,c_n \in C \cup C'$. On the other hand, $\varphi \not \in \Delta$ and so $v_\Delta(\varphi) \in \{\laF,\laf\}$. If $\varphi$ is a closed formula then $v_\Delta(\varphi[\vec{x}/\vec{c}]) \in \{\laF,\laf\}$ for every $c_1,\ldots,c_n \in C \cup C'$. Otherwise, let  $x_1,\ldots,x_n$ be the list of all the variables occurring free in $\varphi$,  and let $\varphi_0= \forall x_1 \cdots \forall x_{n-1} \varphi$. By \axq, $\forall x_n\varphi_0 \neq \Delta$ and so there exists a constant $c_n \in C'$ such that $\varphi_0[x_n/c_n] \not\in \Delta$, since $\Delta$ is a $C'$-Henkin theory in \tms. Let $\varphi_1= \forall x_1 \cdots \forall x_{n-2} \varphi[x_n/c_n]$. By the same reasoning, there exists a constant $c_{n-1} \in C'$ such that $\varphi_1[x_{n-1}/c_{n-1}] \not\in \Delta$. Continuing with this reasoning inductively, we finally found constants $c_1,\ldots,c_n \in C'$ such that $\varphi[\vec{x}/\vec{c}] \not\in \Delta$. This means that $v_\Delta([\vec{x}/\vec{c}])  \in \{\laF,\laf\}$. By Definition~\ref{Tm-sem}, this implies that $\Gamma \not\models_{\tms}\varphi$.
\end{proof}


\section{Analytic Tableaux} \label{secTableaux}
	
In this section, tableau systems for the four-valued non-deterministic modal systems presented in the previous sections will be presented.
We will start by introducing in the first subsection a tableau system for the modal logics {\bf Tm} and \tms. 
In the second subsection, detailed proof of the completeness of the tableau system for \tms\ will be given. Finally, in Subsection~\ref{tableaux-etc} we will present the rules of the respective tableau systems for the logics \sqms\ and \scms, without showing the respective completeness of the method. Indeed, the proof of completeness is very similar to the case of \tms, so we decided to spare the reader the tedious work of accompanying repetitive demonstrations.

\subsection{Tableaux for {\tm} and \tms}

We will now describe an efficient proof procedure for \tm\ and \tms\ based on  \emph{analytic tableaux}. The present approach was adapted from~\cite{smul:1968} and its generalization to many-valued logics introduced in~\cite{car:87}.\footnote{Based on the ideas proposed in the present paper, in~\cite{con:tol:21}  were introduced $(n+2)$-valued tableau systems for da Costa's paraconsistent logics $C_n$. General approaches to tableau proof systems for finite non-deterministic matrices can be found in~\cite{Pawlowski} and~\cite{Gratz:21}.} Let $\phi$ be a formula and let $\laL$  indistinctly denote any truth value \laT, \lat, \laf\ or \laF;  thus $\laL\lab \varphi$ is a \emph{signed} formula.

Consider the following tableau rules for \tm:

$$
\begin{array}{llll}
	\displaystyle \frac{\laT\lab \neg \varphi}{\laF\lab \varphi} & \hspace{10mm} \displaystyle \frac{\lat\lab  \neg \varphi}{\laf\lab  \varphi}  & \hspace{10mm} \displaystyle \frac{\laf\lab  \neg \varphi}{\lat\lab  \varphi} & \hspace{10mm} \displaystyle \frac{\laF\lab  \neg \varphi}{\laT\lab  \varphi} \\[2mm]
	&&\\[2mm]
\end{array}
$$

$$
\begin{array}{llll}
	\displaystyle \frac{\laT\lab  \Box \varphi}{\laT\lab  \varphi} & \hspace{6mm} \displaystyle \frac{\lat\lab  \Box \varphi}{\laT\lab  \varphi}  & \hspace{6mm} \displaystyle \frac{\laf\lab \Box \varphi}{\lat\lab  \varphi \mid \laf\lab  \varphi \mid \laF\lab  \varphi} & \hspace{6mm} \displaystyle \frac{\laF\lab  \Box \varphi}{\lat\lab  \varphi \mid \laf\lab  \varphi \mid \laF\lab  \varphi} \\[2mm]
	&&\\[2mm]
\end{array}
$$

$$
\displaystyle \frac{\,\laT\lab  (\varphi \to \psi)\,}{\laF\lab  \varphi \mid \lat{:} \varphi, \ \lat{:} \psi \mid \laf\lab  \varphi, \ \lat\lab  \psi
	\mid \laf\lab  \varphi, \ \laf\lab  \psi \mid \laT\lab  \psi}$$
\

$$ \displaystyle \frac{\,\lat\lab  (\varphi \to \psi)\,}{\laT\lab  \varphi, \ \lat\lab  \psi \mid \lat\lab  \varphi, \ \lat\lab  \psi \mid \laf\lab  \varphi, \ \lat\lab  \psi
	\mid \laf\lab  \varphi, \ \laf\lab  \psi }
$$

\

$$
\begin{array}{ll}
	\displaystyle \frac{\laf\lab (\varphi \to \psi)}{\laT\lab \varphi, \ \laf\lab \psi \mid \lat\lab \varphi, \ \laf\lab \psi \mid \lat\lab \varphi, \ \laF\lab \psi} & \hspace{6mm} \displaystyle \frac{\,\laF\lab (\varphi \to \psi)\,}{\laT\lab \varphi, \ \laF\lab \psi} \\[2mm]
	&\\[2mm]
\end{array}
$$

\noindent It should be clear that the rules above are directly obtained from the definition of the multioperators in the Nmatrix $\mathcal{M}({\bf Tm})$. This methodology is analogous to the tableau rules obtained in~\cite{smul:1968} from the deterministic two-valued semantics for classical logic, and its extension to tableau systems generated by deterministic finite-valued semantics proposed in~\cite{car:87}.

A branch of a tableau for \tm\ generated by a finite set of signed formulas is said to be {\em closed} if it contains two signed formulas $\laL\lab \varphi$ and $\laL'\lab \varphi$ such that $\laL\neq\laL'$.

The signed tableau rules of \tms\ consist of those for \tm, plus rules for dealing with the quantifiers, to be described below. Along this section, $\Theta$ will denote any  predicate signature, while $\bar C=\{c_n \ : \ n \geq 1\}$ will denote an infinite denumerable set of constants disjoint with $C$. The signature obtained from $\Theta$ by adding the new set of constants $\bar C$ will be denoted by $\Theta(\bar C)$. From now on, we will consider signed formulas of the form $\laL\lab\varphi$, where $\varphi$ is a closed formula over $\Theta(\bar C)$.

\begin{defi}  \label{valsat}
	Let  $\mathfrak{A}= \langle U, \cdot^\mathfrak{A} \rangle$ be a four-valued modal structure over $\Theta(\bar C)$ (recall Definition~\ref{structure}) such that $(C \cup \bar{C})^\mathfrak{A}=U$. Given a \tms-valuation $v$ over $\mathfrak{A}$, we say that a signed formula $\laL\lab\varphi$ is {\em true} in $v$ if $v(\varphi)=\laL$; otherwise, it is {\em false}  in $v$.  If $\varphi$ is a formula over $\Theta(\bar C)$ in which $x$ is the unique variable possibly occurring free, then the closed formula $\varphi[x/c]$ will be denoted by $\varphi(c)$.
\end{defi}

\begin{obs} \label{obs-valuations} \ \\
	(1) Note that, by Remark~\ref{diag-cons-rem}, the structure  $\mathfrak{A}_U$ over $\Theta_U=\Theta(\bar U)$ obtained from  $\mathfrak{A}$ as in Definition~\ref{diagram} is such that $(C \cup \bar{U})^{\mathfrak{A}_U}=U$. Then, the kind of structures considered in Definition~\ref{valsat} are enough to analyze the logic \tms, since $(\mathfrak{A}_U)_U=\mathfrak{A}_U$. This fact will be used in the proof of soundness of the tableau system for \tms\ (see Theorem~\ref{sound-tableaux} below).\\
	(2) Using the previous notation, and from the tables defining the universal quantifier, we obtain the following, for every closed formula over $\Theta(\bar C)$ of the form $\forall x \varphi$: 
	
	\begin{itemize}
		\item[-] $\laT\lab\forall x \varphi$ is true in $v$ iff  $\laT\lab \varphi(c)$ is true in $v$, for every $c \in C \cup \bar{C}$;
		\item[-] $\lat\lab\forall x \varphi$ is true in $v$ iff  $\lat\lab \varphi(c)$ is true in $v$ for some $c \in C \cup \bar{C}$ and, for every $c' \in (C \cup \bar{C})\setminus\{c\}$, either $\laT\lab \varphi(c')$ is true in $v$ or $\lat\lab \varphi(c')$ is true in $v$;
		\item[-] $\laf\lab\forall x \varphi$ is true in $v$ iff  $\laf\lab \varphi(c)$ is true in $v$ for some $c \in C \cup \bar{C}$,  and $\laF\lab \varphi(c')$ is false in $v$ for every $c' \in C \cup \bar{C}$;
		\item[-] $\laF\lab\forall x \varphi$ is true in $v$ iff  $\laF\lab \varphi(c)$ is true in $v$ for some $c \in C \cup \bar{C}$.
	\end{itemize}
\end{obs}

\noindent
By  Remark~\ref{obs-valuations},  when defining a set of tableau rules for \tms, all the rules for quantifiers (with exception of  $\laF\lab\forall x \varphi$) will be {\em reusable}, that is, they can be potentially used with all the constants. The rule for  $\laF\lab\forall x \varphi$ can be used just one time and with a fresh constant, just like happens with tableaux for {\bf CL*} for the signed formulas $\laF(\forall x \varphi)$ and $\laT(\exists x \varphi)$ (see \cite{smul:1968}).

The previous considerations lead us to the following tableau rules for dealing with quantifiers in \tms:

$$
\begin{array}{ll}
	(\laT\forall) \ \displaystyle \frac{\laT\lab \forall x \varphi}{\laT\lab \varphi(c)}  & \hspace{1.5cm} (\laF\forall) \ \displaystyle \frac{\laF\lab \forall x \varphi}{\laF\lab \varphi(c)} \\[2mm]
	&\\[2mm]
\end{array}$$
$$
\begin{array}{ll}
	(\lat\forall) \ \displaystyle \frac{\lat\lab \forall x \varphi}{\lat\lab \varphi(c), \ \lat\lab \varphi(c') \mid \lat\lab \varphi(c), \ \laT\lab \varphi(c')}  &\\[2mm]
	&\\[2mm]
\end{array}$$
$$
\begin{array}{ll}
	(\laf\forall) \ \displaystyle \frac{\laf\lab \forall x \varphi}{\laf\lab \varphi(c), \ \laf\lab \varphi(c') \mid \laf\lab \varphi(c), \ \laT\lab \varphi(c') \mid \laf\lab \varphi(c), \ \lat\lab \varphi(c')}\\[2mm]
\end{array}
$$

\noindent 
{\bf Provisos:} 
\begin{enumerate}
	\item 
	In $(\laT\forall)$, $c$ can be any constant.  This rule is reusable, that is, it can be used several times with different constants on each branch in which the antecedent of the rule appears.
	\item In $(\lat\forall)$, $c$  must be a constant that has not yet appeared in the branch, and $c'$ can be any constant different from $c$.  This rule is reusable, that is, it can be used several times with different constants on each branch in which the antecedent of the rule appears, in the following sense. After branching when apply the rule for the first time, each of the two branches can split into two new branches: the left-side new branch contains the signed formula $\lat\lab \varphi(c'')$, while the right-side branch contains the signed formula $\laT\lab \varphi(c''')$ for any $c''$ and $c'''$   different from $c$.
	\item In $(\laf\forall)$, $c$  must be a constant that has not yet appeared in the branch, and $c'$ can be any constant different from $c$.  This rule is reusable, that is, it can be used several times with different constants on each branch in which the antecedent of the rule appears, in the following sense. After branching when apply the rule for the first time, each of the three branches can splits into three new branches: the first new branch (from left to right) contains the signed formula $\laf\lab \varphi(c'')$, the second new branch contains the signed formula $\laT\lab \varphi(c''')$, and the third new branch  contains the signed formula $\lat\lab \varphi(c'''')$,  for any $c''$,  $c'''$  and $c''''$  different from $c$.
	\item In $(\laF\forall)$, $c$ must be a constant that has not yet appeared in the branch. This rule can be used only once on each branch in which the antecedent of the rule appears.
\end{enumerate}

\begin{defi} \label{closed} A branch of a tableau for \tms\ generated by a signed formula is said to be {\em closed} if it contains two signed formulas $\laL\lab \varphi$ and $\laL'\lab \varphi'$ such that $\varphi \sim \varphi'$ (recall Definition~\ref{variant}),  and $\laL\neq\laL'$. In particular, a branch is closed  if it contains two signed formulas $\laL\lab \varphi$ and $\laL'\lab \varphi$ such that $\laL\neq\laL'$.\footnote{Since, by definition, $\varphi \sim \varphi$ for every $\varphi$.}  A tableau is {\em closed} if any branch is closed. 
\end{defi}

\begin{defi} A closed formula $\varphi$ over $\Theta$ is said to be {\em provable} by tableaux in \tms, denoted by  $\models_{\mathcal{T}(\tms)} \varphi$, if there exists a closed tableau in \tms\ starting from $\laL\lab \varphi$ for every $\laL \in \{\laF, \laf\}$. Given a finite set $\Gamma=\{\gamma_1,\ldots,\gamma_n\}$ of closed formulas over $\Theta$, we say that $\varphi$ is {\em provable from $\Gamma$} by tableaux in \tms, denoted by  $\Gamma \models_{\mathcal{T}(\tms)} \varphi$, if the closed formula $(\gamma_1 \to (\gamma_2 \to \ldots \to(\gamma_n \to \varphi) \ldots ))$ is provable by tableaux in \tms.
\end{defi}

\noi
To prove the soundness of the tableau system for \tms, some definitions are required.

\begin{defi}  \label{valsat1}
	Let  $\mathfrak{A}= \langle U, \cdot^\mathfrak{A} \rangle$ be a four-valued modal structure over $\Theta(\bar C)$ such that $(C \cup \bar{C})^\mathfrak{A}=U$, and let $v$ be  a \tms-valuation over $\mathfrak{A}$.  We say that a branch $\theta$ of a tableau for \tms\ is {\em true under $v$}, or $v$ {\em satisfies} $\theta$, if every signed formula occurring in $\theta$ is true in $v$. A tableau $\F$ for \tms\ is said to be {\em true under $v$}, or $v$ {\em satisfies} $\F$,  if some branch of $\F$ is true under $v$.
\end{defi}

\begin{obs} \label{obs-lema} \ \\
	(1) Observe that, by the previous definitions, a closed branch of a tableau is unsatisfiable. Hence,  any closed tableau is unsatisfiable.\\
	(2) If $\varphi$ is a non-atomic formula, there is exactly one tableau rule, say $R$, appliable to a signed formula of the form $\laL(\varphi)$. It is straightforward to see that if a valuation $v$ satisfies  $\laL(\varphi)$, then it also satisfies all the formulas of at least one of the branches resulting from the application of such rule $R$ to $\laL(\varphi)$.
\end{obs}

\noi  Assume that $\varphi$ is a closed sentence such that $\not\models_\tms \varphi$. From this, and taking into account part~(1) of Remark~\ref{obs-valuations}, there exists some  structure $\mathfrak{A}$ over $\Theta(\bar C)$ with $(C \cup \bar{C})^\mathfrak{A}=U$, as well as a \tms-valuation $v_0$ over it such that $v_0(\varphi) \in \{\laF,\laf\}$. That is, $v_0$ satisfies  $\laL\lab \varphi$ for some $\laL \in \{\laF, \laf\}$.  Now, suppose that $\F$ is a completed tableau in \tms\  starting from $\laL\lab \varphi$. By definition, $\F$ is obtained from a finite sequence of tableaux $\F_0,\ldots,\F_n=\F$, where $\F_0=\laL\lab \varphi$. Suppose that $v$ is a \tms-valuation such that $v$ satisfies  $\F_k$. From part~(2) of Remark~\ref{obs-lema}, it is easy to see that $v$ also satisfies $\F_{k+1}$, for every $0 \leq k \leq n-1$. In particular, this property holds for the valuation $v_0$. Since $v_0$ satisfies $\F_0$, it follows that $v_0$ satisfies $\F$. Hence, by part~(1) of Remark~\ref{obs-lema}, $\F$ cannot be closed. In other words, every completed tableau for  $\laL\lab \varphi$ is open, for some  $\laL \in \{\laF, \laf\}$. This means that  $\not\models_{\mathcal{T}(\tms)} \varphi$. Equivalently: $\models_{\mathcal{T}(\tms)} \varphi$ implies that $\models_\tms \varphi$. This lead us to the following result:

\begin{teo} [Soundness of tableaux for \tms] \label{sound-tableaux} Let $\Gamma \cup\{\varphi\}$ be a finite set of closed formulas over $\Theta$. If $\Gamma \models_{\mathcal{T}(\tms)} \varphi$ \ then \ $\Gamma \models_\tms \varphi$.
\end{teo}
\begin{proof} Taking into consideration the definition of $\models_{\mathcal{T}(\tms)}$, as well as the fact that $\models_\tms$ satisfies the deduction metatheorem for sentences, it is enough to prove the result for $\Gamma=\emptyset$. But it follows from the considerations above.
\end{proof}

\subsection{Completeness of the tableau system for \tms}

In this subsection, the proof of completeness of the tableau system introduced for \tms\ will be obtained. As in the proof for first-order classical logic (see~\cite[Ch.~V, \S 3]{smul:1968}),  a suitable adaptation to the present framework of the notion of Hintikka sets will be useful to our purposes.

\begin{defi} \label{hintikka}
	A set $\Gamma$ of signed formulas over $\Theta(\bar C)$ is said to be a {\em Hintikka set} for \tms\ in  the universe $C \cup \bar C$ if the following holds:
	\begin{enumerate}
		\item If $\laL\lab\varphi$ and $\laL'\lab\varphi'$ belong to $\Gamma$ such that $\varphi \sim \varphi'$, then $\laL=\laL'$.
		In particular, if $\laL\lab\varphi$ and $\laL'\lab\varphi$ belong to $\Gamma$ then $\laL=\laL'$.
		\item If $\laL\lab \neg \varphi$ belongs to $\Gamma$ then $\neg\laL\lab \varphi$ belongs to $\Gamma$, where $\neg\laL$ denotes the unique element of the set $\tilde{\neg} \, \laL$.
		\item If $\laT\lab  \Box \varphi$ or $\lat\lab  \Box \varphi$ belong to $\Gamma$ then $\laT\lab  \varphi$ belongs to $\Gamma$.
		\item If $\laF\lab  \Box \varphi$ or  $\laf\lab  \Box \varphi$ belong to $\Gamma$ then  $\laL\lab \varphi$ belongs to $\Gamma$ for a unique $\laL \in \{\lat, \laf, \laF\}$.
		\item If $\laT\lab  (\varphi \to \psi)$ belongs to $\Gamma$ then: either  $\laF\lab  \varphi$   belongs to $\Gamma$, or $\laT\lab  \psi$ belongs to $\Gamma$, or $\lat\lab \varphi$ and $\lat\lab \psi$  belong to $\Gamma$, or  $\laf\lab  \varphi$ and $\lat\lab  \psi$  belong to $\Gamma$, or $\laf\lab  \varphi$ and $\laf\lab  \psi$  belong to $\Gamma$.
		\item If $\lat\lab  (\varphi \to \psi)$ belongs to $\Gamma$ then: either $\laT\lab  \varphi$ and $\lat\lab  \psi$  belong to $\Gamma$, or  $\lat\lab  \varphi$ and $\lat\lab  \psi$  belong to $\Gamma$, or  $\laf\lab  \varphi$ and $\lat\lab  \psi$  belong to $\Gamma$, or  $\laf\lab  \varphi$ and $\laf\lab  \psi$  belong to $\Gamma$.
		\item If $\laf\lab (\varphi \to \psi)$ belongs to $\Gamma$ then: either $\laT\lab \varphi$ and $\laf\lab \psi$  belong to $\Gamma$, or  $\lat\lab \varphi$ and $\laf\lab \psi$  belong to $\Gamma$, or  $\lat\lab \varphi$ and $\laF\lab \psi$  belong to $\Gamma$.
		\item If $\laF\lab (\varphi \to \psi)$ belongs to $\Gamma$ then $\laT\lab \varphi$ and $\laF\lab \psi$ belong to $\Gamma$.
		\item If $\laT\lab\forall x \varphi$ belongs to $\Gamma$ then  $\laT\lab \varphi(c)$ belongs to $\Gamma$, for every $c \in C \cup \bar C$;
		\item If $\lat\lab\forall x \varphi$ belongs to $\Gamma$ then $\lat\lab \varphi(c)$ belongs to $\Gamma$ for some $c \in C \cup \bar C$ and, for every $c' \in (C \cup \bar C)\setminus\{c\}$, either $\laT\lab \varphi(c')$ belongs to $\Gamma$ or $\lat\lab \varphi(c')$ belongs to $\Gamma$;
		\item If $\laf\lab\forall x \varphi$ belongs to $\Gamma$ then $\laf\lab \varphi(c)$ belongs to $\Gamma$ for some $c \in C \cup \bar C$ and, for every $c' \in (C \cup \bar C)\setminus\{c\}$: either $\laT\lab \varphi(c')$ belongs to $\Gamma$, or $\lat\lab \varphi(c')$ belongs to $\Gamma$, or $\laf\lab \varphi(c')$ belongs to $\Gamma$, and $\laF\lab \varphi(c')$ does not belong to $\Gamma$;
		\item If $\laF\lab\forall x \varphi$ belongs to $\Gamma$ then $\laF\lab \varphi(c)$ belongs to $\Gamma$ for some $c \in C \cup \bar C$.
	\end{enumerate}
\end{defi}

\begin{defi} \label{complexity} Let $\Theta$ be a  predicate signature. The {\em complexity} $\lac(\varphi)$ of a formula $\varphi \in For(\Theta)$ is defined recursively as follows: $\lac(\varphi)=0$ if $\varphi$ is atomic; $\lac(\neg\varphi)=\lac(\Box \varphi)=\lac(\forall x \varphi)=\lac(\varphi)+1$; and $\lac(\varphi \to \psi)=\lac(\varphi)+\lac(\psi)+1$.
\end{defi}

\begin{lema} \label{extenv} Let $\Gamma$ be a Hintikka set for \tms\ in the universe $C \cup \bar C$, and let $\mathfrak{A}= \langle U, \cdot^\mathfrak{A} \rangle$ be a four-valued modal structure  over $\Theta(\bar C)$  such that $(C \cup \bar{C})^\mathfrak{A}=U$ (hence, by Definition~\ref{diagram}, $\Theta(\bar C)_U=\Theta(\bar C)$ and  $\mathfrak{A}_U= \mathfrak{A}$). Let $\Gamma_0 = \{ \varphi \in Sen(\Theta(\bar C)) \ : \ \laL\lab\varphi \in \Gamma\}$, and let $v:\Gamma_0 \to \{\laT,\lat,\laf,\laF\}$ be a function defined as follows: $v(\varphi)=\laL$ iff  $\laL\lab\varphi \in \Gamma$.	Then, $v$ is well-defined and there exists a \tms-valuation $\bar v:Sen(\Theta(\bar C)) \to \{\laT,\lat,\laf,\laF\}$ over $\mathfrak{A}$ extending $v$.
\end{lema}
\begin{proof}  By item~1 of Definition~\ref{hintikka}, $v$ is a well-defined function. Moreover, if $\varphi \sim \varphi'$ and both belong to $\Gamma_0$ then $v(\varphi)=v(\varphi')$. In particular, if $\varphi'=Q_1 x_1 \ldots Q_k x_k\psi$ where $k \geq 0$ and $Q_i \in \{\forall,\exists\}$\footnote{If $k=0$ then the sequence of quantifiers is empty. Recall that $\exists x \gamma$ stands for $\neg\forall x \neg \gamma$.} such that $\psi$ is closed (hence $\psi \sim \varphi$) then $\psi \in \Gamma_0$ and $v(\varphi')=v(\psi)=v(\varphi)$, taken into account that $\tilde{\neg}$ is deterministic and the value of $\tilde{\neg}\tilde{\neg}\, a$ is precisely $a$. This fact will be used along this proof. 
	
	The definition of $\bar v(\varphi)$ will be done by induction on the complexity $\lac(\varphi)$ of $\varphi \in Sen(\Theta(\bar C))$. Moreover, at each step it will defined $\bar v(\varphi'):=\bar v(\varphi)$ for every $\varphi' \in Sen(\Theta(\bar C))$ such that  $\varphi \sim \varphi'$ and  $\lac(\varphi') \geq  \lac(\varphi)$. To this respect observe that, if $Q_1 x_1 \ldots Q_k x_k \psi \sim \varphi$, $Q_i \in \{\forall,\exists\}$, $\psi$ is a closed sentence and $\bar{v}(\varphi)$ was already defined then we can define $\bar{v}(Q_1 x_1 \ldots Q_k x_k \psi)=\bar{v}(\psi)=\bar{v}(\varphi)$ in a coherent way.
	
	Thus, assume first that $\varphi$ is an atomic closed formula in $Sen(\Theta(\bar C))$. If $\varphi \in \Gamma_0$ then put $\bar v(\varphi)=v(\varphi)$; if  $\varphi \notin \Gamma_0$ then define $\bar v(\varphi)$ arbitrarily (for instance, $\bar v(\varphi)=\laT$). For every $\varphi' \in Sen(\Theta(\bar C))$ such that $\varphi \sim \varphi'$ and  $\lac(\varphi') \geq  \lac(\varphi)$ define $\bar v(\varphi')=v(\varphi)$. As observed above, if any of such $\varphi'$ is in $\Gamma_0$ then $\bar v(\varphi')=v(\varphi')$.  This completes the definition of $\bar v$ for atomic sentences and all of its variants). 
	
	Assume that $\bar v$ was defined for every sentence $\psi$ in $Sen(\Theta(\bar C))$ such that $\lac(\psi)\leq n$ (as well as for all of its variants with arbitrary complexity), by extending $v$ and satisfying the clauses  for valuation (induction hypothesis --- IH). Namely: $\bar v(\psi)=v(\psi)$ if $\psi \in \Gamma_0$ and $\lac(\psi)\leq n$; if $\psi \sim \gamma$ then $\bar v(\gamma)=\bar v(\psi)$; $\bar v (\neg \psi) \in \tilde{\neg} \, \bar v(\psi)$; $\bar v (\Box \psi) \in \tilde{\Box}_1 \, \bar v(\psi)$; $\bar v(\forall x  \psi) \in \tilde{\forall}_4^d\big(X(\psi,x,\bar v)\big)$, where $X(\psi,x,\bar v) = \big\{\bar v(\psi[x/c]) \ : \  c \in C \cup \bar C\big\}$;\footnote{In particular, if $x$ does not occur free in $\psi$ then $\bar v(Q x \psi)=\bar v(\psi)$ for $Q \in \{\forall,\exists\}$, since $ \tilde{\forall}_4^d\big(\{\laL\}\big)=\{\laL\}$ for every $\laL$ and $\tilde{\neg}$ is deterministic such that the value of $\tilde{\neg}\tilde{\neg}\, a$ is precisely $a$. This is coherent with the fact that $Q x \psi \sim \psi$.} and $\bar v(\psi \to \gamma) \in \bar v(\psi) \,\tilde{\to}\, \bar v(\gamma)$. Now, consider  a formula $\varphi$ such that $\lac(\varphi)=n+1$. We will show how to define $\bar v(\varphi)$ as well as $\bar v(\psi)$ for every $\psi \sim \varphi$ with $\lac(\psi) \geq n+1$. \\[1mm]
	{\bf Case} $\varphi=\neg\psi$. Then, $\bar v(\psi)$ was already defined. Define now $\bar v(\varphi)$ as being the unique element of $\tilde{\neg} \, \bar v(\psi)$. Note that, if $\varphi \in \Gamma_0$ then $\psi \in \Gamma_0$, $\bar v(\psi)=v(\psi)$ (by (IH) and  $\bar v(\varphi)=v(\varphi)$, by Definition~\ref{hintikka}. If $\gamma \sim \varphi$ with $\lac(\gamma)\geq n+1$ define $\bar v(\gamma)=\bar v(\varphi)$. Observe that $\gamma=Q_1 x_1 \ldots Q_k x_k\neg\delta$ where $k \geq 0$, $Q_i \in \{\forall,\exists\}$ and $\delta \sim \psi$ and so $\bar v(\delta)=\bar v(\psi)$, hence this definition is coherent.\\[1mm]
	{\bf Case} $\varphi=\Box\psi$. Then, $\bar v(\psi)$ was already defined. If $\varphi \in \Gamma_0$ then $\psi \in \Gamma_0$ and $v(\varphi) \in \tilde{\Box}_1 \, v(\psi)$, by Definition~\ref{hintikka}; in this case define $\bar v(\varphi)=v(\varphi)$. Now, if $\varphi \notin \Gamma_0$ but $\gamma \in \Gamma_0$ for some $\gamma \sim \varphi$, then $\gamma=Q_1 x_1 \ldots Q_k x_k\Box\gamma'$ such that $k \geq 0$, $Q_i \in \{\forall,\exists\}$ and $\gamma' \sim \psi$. In this case $\gamma' \in \Gamma_0$, by Definition~\ref{hintikka}, and so $v(\gamma')=v(\psi)$ such that $v(\gamma) \in \tilde{\Box}_1 \, v(\gamma')$. Define $\bar v (\varphi)=v(\gamma')$. If $\gamma \notin \Gamma_0$ for every $\gamma \sim \varphi$ define  $v(\varphi) \in \tilde{\Box}_1 \, v(\psi)$ arbitrarily.
	Finally, define $\bar v(\gamma)=\bar v(\varphi)$ for every $\gamma$ such that $\gamma \sim  \varphi$ and $\lac(\gamma) \geq n+1$ (by observing that $\gamma=Q_1 x_1 \ldots Q_k x_k\Box\gamma'$ where $k \geq 0$, $Q_i \in \{\forall,\exists\}$ and $\gamma' \sim \psi$, hence this definition is coherent). \\[1mm]
	{\bf Case} $\varphi=\forall x\psi$. Note that all the values in the set  $X=\{\bar v(\psi[x/c]) \ : \ c \in C \cup \bar{C}\}$ where already defined. As observed above, if $x$ is not free in $\psi$ then $\bar v(\varphi)$ was already defined and $\bar v(\varphi)=\bar v(\psi)$, since $\varphi \sim \psi$. Now, suppose that $x$ occurs free in $\psi$. We have several subcases to analyze.
	\begin{itemize}
		\item {\bf Case} $\varphi \in \Gamma_0$. There are two subcases to analyze:
		\begin{itemize}
			\item {\bf Case} $v(\varphi)=\laF$. Then $\psi[x/c] \in \Gamma_0$  for some $c \in C \cup \bar{C}$ such that $\bar v(\psi[x/c])=v(\psi[x/c])=\laF$. By defining $\bar v (\varphi)=v(\varphi)$ we get that $\bar v(\varphi) \in\tilde{\forall}_4^d\big(X\big)$.
			\item {\bf Case} $v(\varphi)\neq \laF$. Then $\psi[x/c] \in \Gamma_0$ and $\bar v(\psi[x/c])=v(\psi[x/c])$,   for every $c \in C \cup \bar{C}$, hence $v(\varphi) \in \tilde{\forall}_4^d\big(X\big)$, by Definition~\ref{hintikka}. In this case define $\bar v(\varphi)=v(\varphi)$, hence $\bar v(\varphi) \in\tilde{\forall}_4^d\big(X\big)$.
		\end{itemize}
		
		\item {\bf Case} $\varphi \notin \Gamma_0$. By (IH), $\bar v(\psi[x/c])=v(\psi[x/c])$ for every $c \in C \cup \bar{C}$ such that $\psi[x/c] \in \Gamma_0$. Observe that, if $\gamma \sim \varphi$, then $\gamma=Q_1 x_1 \ldots Q_k x_k\delta$, where $k \geq 1$, $Q_j \in \{\forall,\exists\}$ and $\delta \sim \psi[x/x_i]$ for some $1 \leq i \leq k$ such that $x_i$ is free for $x$ in $\psi$, and $x_i$ is the only variable occurring free in $\delta$. Then $\bar v(\delta[x_i/c])=\bar v(\psi[x/c])$ for every $c \in C \cup \bar{C}$, by (IH). Moreover,  $\bar v(\psi[x/c])=v(\delta[x_i/c])$ for every $c \in C \cup \bar{C}$ such that $\delta[x_i/c] \in \Gamma_0$. We have two subcases to analyze:
		\begin{itemize}
			\item There is some $\gamma \in \Gamma_0$ such that $\gamma \sim \varphi$. Then $\gamma=Q_1 x_1 \ldots Q_k x_k\delta$, where  $Q_j \in \{\forall,\exists\}$, $\delta \sim \psi[x/x_i]$ for some $1 \leq i \leq k$ and  $\bar v(\psi[x/c])=v(\delta[x_i/c])$ for every $c \in C \cup \bar{C}$ such that $\delta[x_i/c] \in \Gamma_0$, as observed above. Then $v(\gamma)$, which is given according to Definition~\ref{hintikka}, is such that $v(\gamma) \in \tilde{\forall}_4^d\big(X\big)$.  In this case define  $\bar v(\varphi)=v(\gamma)$, and so $v(\varphi) \in \tilde{\forall}_4^d\big(X\big)$.
			
			\item For every $\gamma \in \Gamma_0$ is not the case that $\gamma \sim \varphi$. In this case define  $\bar v(\varphi) \in \tilde{\forall}_4^d\big(X\big)$ arbitrarily.
		\end{itemize}
	\end{itemize}
	
	\noindent Finally, if $\gamma \sim \varphi$ with $\lac(\gamma)\geq n+1$ define $\bar v(\gamma)=\bar v(\varphi)$. As observed above, $\gamma=Q_1 x_1 \ldots Q_k x_k\delta$, where $k \geq 1$, $Q_j \in \{\forall,\exists\}$ and $\delta \sim \psi[x/x_i]$ for some $1 \leq i \leq k$. Then $\bar v(\delta[x_i/c])=\bar v(\psi[x/c])$ for every $c \in C \cup \bar{C}$, by (IH), and so the value $\bar v(\gamma)$ is coherent with Definition~\ref{Tm-sem}.\\[1mm]
	{\bf Case} $\varphi=\gamma \to \psi$. There are two main cases to analyze:
	\begin{itemize}
		\item {\bf Case} $\varphi \in \Gamma_0$. It produces two subcases:
		\begin{itemize}
			\item {\bf Case} $v(\varphi)=\laT$. Then:  either $\psi\in \Gamma_0$ and $v(\psi)=\laT$, or  $\gamma\in \Gamma_0$ and $v(\gamma)=\laF$. Given that $\laL \,\tilde{\to}\, \laT=\laF \,\tilde{\to}\, \laL=\{\laT\}$ for every $\laL$ then, by defining $\bar v(\varphi)=v(\varphi)$, we guarantee that $\bar v(\varphi) \in \bar v(\gamma) \,\tilde{\to}\, \bar v(\psi)$, by (IH).
			\item {\bf Case} $v(\varphi)\neq \laT$. Then $\gamma \in \Gamma_0$, $\psi \in \Gamma_0$, $\bar v(\gamma)=v(\gamma)$ and $\bar v(\psi)=v(\psi)$, by (IH) and by Definition~\ref{hintikka}. Then, by defining $\bar v(\varphi)=v(\varphi)$ we guarantee that $\bar v(\varphi) \in \bar v(\gamma) \,\tilde{\to}\, \bar v(\psi)$, by  Definition~\ref{hintikka}.
		\end{itemize}
		
		\item {\bf Case} $\varphi \notin \Gamma_0$.  We have two subcases to analyze:
		\begin{itemize}
			\item There is some $\delta \in \Gamma_0$ such that $\delta \sim \varphi$. Then $\delta=Q_1 x_1 \ldots Q_k x_k(\gamma' \to \psi')$, where $\gamma' \sim \gamma$,  $\psi' \sim \psi$, $Q_i \in \{\forall,\exists\}$ and $k \geq 0$. By (IH),  $\bar v(\gamma')=\bar v(\gamma)$ and  $\bar v(\psi')=\bar v(\psi)$. According to Definition~\ref{hintikka}, and as observed in the previous case (where $\varphi \in \Gamma_0$), we have that $\gamma' \to \psi' \in \Gamma_0$, $\alpha \in \Gamma_0$ for some $\alpha \in \{\gamma', \psi'\}$ and $\bar v(\alpha)=v(\alpha)$ if $\alpha \in \Gamma_0 \cap \{\gamma', \psi'\}$. Moreover, $v(\delta) =v(\gamma' \to \psi')$. Then,  by defining $\bar v(\varphi)=v(\delta)$ we guarantee that $\bar v(\varphi) \in \bar v(\gamma) \,\tilde{\to}\, \bar v(\psi)$, by means of an analysis similar to the previous case.
			
			\item For every $\gamma \in \Gamma_0$ is not the case that $\gamma \sim \varphi$. In this case define  $\bar v(\varphi) \in \bar v(\gamma) \,\tilde{\to}\, \bar v(\psi)$ arbitrarily. Observe that, if either $\gamma' \in \Gamma_0$ for some $\gamma' \sim \gamma$ or $\psi' \in \Gamma_0$ for some  $\psi' \sim \psi$ then $\bar v(\gamma)=v(\gamma')$ or $\bar v(\psi) =v(\psi')$, respectively, by (IH). Hence this definition is coherent.
		\end{itemize}
	\end{itemize}
	
	\noindent Finally, let $\delta \sim \varphi$ such that $\lac(\delta) \geq n+1$. Then $\delta=Q_1 x_1 \ldots Q_k x_k(\gamma' \to \psi')$, where $\gamma' \sim \gamma$,  $\psi' \sim \psi$, $Q_i \in \{\forall,\exists\}$ and $k \geq 0$. By defining $\bar v(\delta)=\bar v(\varphi)$ and by  (IH), we have that $\bar v(\delta)=\bar v(\gamma' \to \psi') \in \bar v(\gamma') \,\tilde{\to}\, \bar v(\psi')$, as required.
	
	From this construction, it is clear that $\bar v:Sen(\Theta(\bar C)) \to \{\laT,\lat,\laf,\laF\}$ is a \tms-valuation over $\mathfrak{A}$ extending $v$.	 
\end{proof}

\begin{teo} [Hintikka's Lemma for \tms] \label{hintikka-lemma} Let $\Gamma$ be a Hintikka set for \tms\ in the universe $C \cup \bar C$. Then, there is a four-valued modal structure $\mathfrak{A}= \langle U, \cdot^\mathfrak{A} \rangle$  over $\Theta(\bar C)$  where $(C \cup \bar{C})^\mathfrak{A}=U$, and a \tms-valuation $\bar v$ over it such that $\laL\lab\varphi$ is true in $\bar v$  for every $\laL\lab\varphi \in \Gamma$.
\end{teo}
\begin{proof} Let $U=C \cup \bar C$ and let $\mathfrak{A}= \langle U, \cdot^{\mathfrak{A}} \rangle$ be a four-valued modal structure  over $\Theta$ defined as follows: 	
	\begin{itemize}
		\item[-] For each $n$-ary predicate $P$, $P^\mathfrak{A}:U^n \to \{\laT,\lat,\laf,\laF\}$ is defined as follows:
		$P^\mathfrak{A}(a_1,\ldots,a_n)=\laL$ if $\laL\lab P a_1\ldots a_n \in \Gamma$, and it gets an arbitrary value in $\{\laT,\lat,\laf,\laF\}$ otherwise;
		\item[-] For each individual constant $c \in C \cup \bar C$, $c^\mathfrak{A} = c$.
	\end{itemize}
	It is worth observing that $P^\mathfrak{A}$ is well-defined, by item~1 of Definition~\ref{hintikka}. Since $(C \cup \bar C)^\mathfrak{A}=U$ then, by Definition~\ref{diagram}, $\Theta(\bar C)_U=\Theta(\bar C)$ and  $\mathfrak{A}_U= \mathfrak{A}$. Let $\Gamma_0 = \{ \varphi \in Sen(\Theta(\bar C)) \ : \ \laL\lab\varphi \in \Gamma\}$, and let $v:\Gamma_0 \to \{\laT,\lat,\laf,\laF\}$ be a function defined as follows: $v(\varphi)=\laL$ iff  $\laL\lab\varphi \in \Gamma$. By Lemma~\ref{extenv}, $v$ is well-deﬁned, and there exists a \tms-valuation $\bar v$ over $\mathfrak{A}$ extending $v$. That is, $\bar v$ is a \tms-valuation over $\mathfrak{A}$ such that $\laL\lab\varphi$ is true in $\bar v$  for every $\laL\lab\varphi \in \Gamma$.
\end{proof}

\noindent  Finally, to prove the completeness of the tableau system for \tms, the notion of {\em systematic tableaux} proposed by Smullyan for his tableau system for classical first-order logic (see~\cite[p.~59]{smul:1968}) must be adapted to the specific rules of this logic. For technical reasons that will be clear below, besides signed formulas $\laL\lab\varphi$, we will consider {\em marked} signed formulas, which are labeled signed formulas of the form $\lat\lab\forall x\psi\lab[c]$ or $\laf\lab\forall x\psi\lab[c]$ such that $\psi$ is a formula in which $x$ is the only variable (possibly) occurring free, and $c$ is a constant of the signature. 

\begin{defi} \label{systematic}
	Let $\laL'\lab\gamma$ be a signed formula over a signature $\Theta$. Let $C=\{c_1,\ldots,c_n\}$ be the (possibly empty) set of constants occurring in $\gamma$, and consider an infinite denumerable set $\bar C = \{c_{n+1}, c_{n+2}, \ldots\}$ of new constants (observe that $C \cup \bar C$ is considered to be linearly ordered). The procedure for defining a {\em systematic tableau} $\F$ in \tms\ for $\laL'\lab\gamma$, which is a (possibly infinite) tree of signed formulas or marked signed formulas  over $\Theta(\bar C)$ of degree~5,\footnote{Meaning that each node has, at most, 5 child nodes. This is an obvious consequence of the tableau rules defined above.} is defined as follows: \\
	(1) Put the signed formula $\laL'\lab\gamma$ at the beginning of the tree, forming an initial branch $\theta$ of $\F$, thus completing stage 1 of the procedure with a tableau $\F_1$.\\
	(2) Assume that a tableau $\F_n$ (that is, a  tree of degree~5) was already completed at the $n$th stage of the procedure. If $\F_n$ is closed, the procedure stops. If  $\F_n$ is not closed, but every non-atomic signed formula was used on every open branch, the procedure also stops.\footnote{Observe that the procedure cannot stop at this point if a reusable signed formula or a marked signed formula appears in an open branch of $\F_n$, as such an expression can still be used.} Otherwise,  pick a  non-atomic signed formula $\laL\lab\varphi$ or a marked signed formula $\lat\lab\forall x\psi\lab[c]$ or $\laf\lab\forall x\psi\lab[c]$  of minimal level (which means that such expression is located as high as possible) in  the tree $\F_n$  which has not yet been used; having more than one of such unused expressions at the same minimal level of the tree, pick  the leftmost one. Then, extend {\bf every} open branch  $\theta$ containing such occurrence of $\laL\lab\varphi$, $\lat\lab\forall x\psi\lab[c]$ or $\laf\lab\forall x\psi\lab[c]$, as follows (clauses 1-7 refer to $\laL\lab\varphi$, 8 refers to $\lat\lab\forall x\psi\lab[c]$ and~9 refers to $\laf\lab\forall x\psi\lab[c]$):
	\begin{enumerate}
		\item If either $\varphi$ is $\Box\delta$ and $\laL \in \{\laT,\lat\}$, or $\varphi$ is $\neg\delta$,  extend $\theta$ to $(\theta, \, \laL''\lab\delta')$, where $\laL''\lab\delta'$ is the consequence of the respective tableau rule.
		\item If $\varphi$ is $\delta \to \psi$ and $\laL\neq \laF$ extend $\theta$ by 5, 4 or 3 branches (if $\laL$ is $\laT$, $\lat$ or $\laf$, respectively) with the corresponding signed formulas on each branch, according to the specific tableau rule.
		\item If $\laL\lab\varphi$ is $\laF\lab(\delta \to \psi)$, extend $\theta$ to $(\theta, \, \laT\lab\delta, \, \laF\lab\psi)$.
		\item If $\varphi$ is $\forall x\delta$ and $\laL=\laF$,  extend $\theta$ to $(\theta, \, \laF\lab\delta(c))$, where $c$ is the first constant that has not yet appeared on $\theta$.
		\item If $\varphi$ is $\forall x\delta$ and $\laL=\laT$,  extend $\theta$ to $(\theta, \, \laT\lab\delta(c), \, \laT\lab\forall x\delta)$, where $c$ is the first constant such that $\laT\lab\delta(c)$ does not occur on $\theta$.
		\item If $\varphi$ is $\forall x\delta$ and $\laL=\lat$,  extend $\theta$ to $(\theta, \, \lat\lab\delta(c), \, \lat\lab\forall x\delta\lab[c])$, where  $c$ is the first constant that has not yet appeared in $\theta$.
		\item If $\varphi$ is $\forall x\delta$ and $\laL=\laf$,  extend $\theta$ to $(\theta, \, \laf\lab\delta(c), \, \laf\lab\forall x\delta\lab[c])$, where  $c$ is the first constant that has not yet appeared in $\theta$.
		\item If the first non-atomic unused expression is  $\lat\lab\forall x\psi\lab[c]$, extend $\theta$ to the two branches $(\theta, \,  \lat\lab\psi(c'), \, \lat\lab\forall x\psi\lab[c])$ and $(\theta,  \, \laT\lab\psi(c''), \, \lat\lab\forall x\psi\lab[c])$, where $c'$ is the first constant different from $c$ such that $\lat\lab\psi(c')$ does not occur on $\theta$, and $c''$ is the first constant different from $c$ such that  $\laT\lab\psi(c'')$ does not occur on $\theta$.
		\item If the first non-atomic unused  expression is $\laf\lab\forall x\psi\lab[c]$, extend $\theta$ to the three branches  $(\theta, \, \laf\lab\psi(c'), \, \laf\lab\forall x\psi\lab[c])$, $(\theta,  \, \laT\lab\psi(c''), \, \laf\lab\forall x\psi\lab[c])$, and $(\theta,  \, \lat\lab\psi(c'''), \, \laf\lab\forall x\psi\lab[c])$, where $c'$ is the first constant different from $c$ such that $\lat\lab\psi(c')$ does not occur on $\theta$, $c''$ is the first constant different from $c$ such that  $\laT\lab\psi(c'')$ does not occur on $\theta$, and $c'''$ is the first constant different from $c$ such that  $\lat\lab\psi(c''')$ does not occur on $\theta$.
	\end{enumerate} 
	After performing step (2), the corresponding expression of the tree chosen in each of these steps (namely, $\laL\lab\varphi$, $\lat\lab\forall x\psi\lab[c]$  or $\laf\lab\forall x\psi\lab[c]$) is declared to be used, thus concluding  the stage $n+1$ of the procedure.
\end{defi} 

\noindent As in the case of Smullyan's systematic tableau procedure for classical first-order logic, the purpose of repeating an occurrence of a signed formula  $\laT\lab\forall x\psi$ after an instance $\laT\lab\psi(c)$ is to allow their reuse with another constant (given that the original occurrence of $\laT\lab\forall x\psi$ is declared to be used). This procedure guarantees that any instance  $\laT\lab\psi(c)$ will appear in an open branch of a finished systematic tableau in \tms. The same technique is applied to guarantee that the signed formula  $\lat\lab\forall x\psi$ will be reused. However, in this case an initial instance $\lat\lab\psi(c)$ with a new constant $c$ is added, together with the expression $\lat\lab\forall x\psi\lab[c]$. This expression contains the signed formula $\lat\lab\forall x\psi$ to be reused, plus a mark $[c]$ indicating that the rule was used for the first time with the fresh constant $c$. When this rule is reused after this stage (as indicated on item~8 of step~(2)), the tableau splits into two branches: the left-side branch contains an instance $\lat\lab\psi(c')$ which does not occur on $\theta$, although $c' \neq c$ is not necessarily new in the branch, while the right-side branch contains an instance $\laT\lab\psi(c'')$ which does not occur on $\theta$, although $c'' \neq c$ is not necessarily new in the branch. The constant $c$ in the mark informs that $\psi(x)$ cannot be instantiated once again with $c$. Below each of these two formulas, the expression $\lat\lab\forall x\psi\lab[c]$ is repeated on each of the two new branches, allowing new rule reuse (given that the original occurrence of $\lat\lab\forall x\psi\lab[c]$ is declared to be used). This procedure guarantees that, for any constant $c'$,  $\laL\lab\psi(c')$ will appear in an open branch of a finished systematic tableau in \tms\ with a unique label $\laL \in \{\laT,\lat\}$. A similar technique is employed for rule $(\laf\forall)$ ensuring that,  for any constant $c'$,  $\laL\lab\psi(c')$ will appear in an open branch of a finished systematic tableau in \tms\ with a unique label $\laL \neq \laF$.

\begin{defi}
	A {\em finished} systematic tableau in \tms\  is a systematic tableau in \tms\  which is either infinite (hence it contains at least an infinite branch, by K\"onig's lemma), or it is finite but it cannot be extended further employing the systematic procedure described in Definition~\ref{systematic} (that is, on every open branch every non-atomic signed formula was used).
\end{defi}

\begin{obs} \label{rem-finished} It is worth noting that if a finished systematic tableau $\mathcal{F}$ in \tms\ is infinite then the procedure described in Definition~\ref{systematic} for defining it cannot stop in any finite step $k$. Indeed, the tableau $\mathcal{F}_k$ obtained in step $k$ of the definition of $\mathcal{F}$ is finite, since it is a tree of degree~5 and $k$ is finite. Observe that if $\theta$ is a  finite open branch of a finished systematic tableau $\mathcal{F}$ in \tms, then no reusable signed formula or marked signed formula occurs in $\theta$. Otherwise, such an expression would give origin, in a later step of the construction of $\F$, to an unused occurrence in $\theta$  of a reusable signed formula or of a marked signed formula and then $\theta$ could be extended {\em ad infinitum} by the systematic procedure given in Definition~\ref{systematic}, which contradicts the fact that $\theta$ is finite. From the previous considerations, it is clear that every  occurrence of a (reusable or not) non-atomic signed formula or of a marked signed formula in an open branch of a (infinite or not) finished $\mathcal{F}$ was used at some point of the construction of  $\mathcal{F}$.
\end{obs}

\begin{prop} \label{open-hintikka} Let $\laL_0\lab\varphi_0$  be a signed formula over $\Theta$, and let $\bar C$ be as in Definition~\ref{systematic}.
	Let $\theta$ be an open branch of a finished systematic tableau $\mathcal{F}$ in \tms\ for $\laL\lab\varphi$, and let $\Gamma_0$ be the set of signed formulas occurring in  $\theta$  (so, marked signed formulas as $\lat\lab\forall x\psi\lab[c]$  or $\laf\lab\forall x\psi\lab[c]$ occurring in $\theta$ are not included in $\Gamma_0$).
	Then, $\Gamma_0$ is a Hintikka set for \tms\ in  the universe $C \cup \bar C$.
\end{prop}
\begin{proof}
	Since $\theta$ is open then, by Definition~\ref{closed}: if $\laL\lab\varphi$ and $\laL'\lab\varphi'$ belong to $\Gamma_0$ such that $\varphi$ and $\varphi'$ are variant, then $\laL=\laL'$. In particular, if $\laL\lab\varphi$ and $\laL'\lab\varphi$ belong to $\Gamma_0$ then $\laL=\laL'$. This shows that $\Gamma_0$ satisfies clause~1 of Definition~\ref{hintikka}. If $\laL\lab\varphi \in \Gamma_0$ for $\varphi$ of the form $\neg\psi$, $\Box\psi$ or $\gamma \to \psi$ then, by the tableau rules for \tm\ (which are included in the tableau system for \tms), and taking into consideration that $\mathcal{F}$ is a finished systematic tableau, $\laL\lab\varphi$ was used at some stage of the procedure for defining $\theta$ (as observed in Remark~\ref{rem-finished}), hence  it is immediate to see that clauses~2 to~8  of Definition~\ref{hintikka} are fullfilled. If $\laT\lab\forall x\varphi \in \Gamma_0$ then, since $\mathcal{F}$ is a finished systematic tableau and $\theta$ is open, this signed formula was used with all the available constants, as discussed in Remark~\ref{rem-finished}. That is, $\laT\lab\varphi(c) \in \Gamma_0$  for every $c \in C \cup \bar C$, showing that $\Gamma_0$ satisfies clause~9 of Definition~\ref{hintikka}. 
	If $\lat\lab\forall x\varphi \in \Gamma_0$ then, given that $\mathcal{F}$ is a finished systematic tableau and $\theta$ is an open branch, $\theta$ contains $\lat\lab\varphi(c)$ for some constant $c$, and the marked signed formula $\lat\lab\forall x\varphi\lab[c]$ also occurs in $\theta$ (see Remark~\ref{rem-finished}). Since the latter was used with all available constants other than $c$ we have that, for every $c' \in (C \cup \bar C)\setminus\{c\}$: either $\laT\lab \varphi(c')$ belongs to $\Gamma_0$ or $\lat\lab \varphi(c')$ belongs to $\Gamma_0$. This shows that $\Gamma_0$ satisfies clause~10 of Definition~\ref{hintikka}.
	Now, if $\laf\lab\forall x\varphi \in \Gamma_0$ then, given that $\mathcal{F}$ is a finished systematic tableau and $\theta$ is an open branch, $\laf\lab\forall x\varphi$ was used (see Remark~\ref{rem-finished}) and so  $\laf\lab\varphi(c) \in \Gamma_0$  for some $c \in C \cup \bar C$,  plus the marked signed formula $\lat\lab\forall x\varphi\lab[c]$. Given that the latter was used with all available constants other than $c$ it follows that, for every $c' \in (C \cup \bar C)\setminus\{c\}$: either $\laT\lab \varphi(c')$ belongs to $\Gamma_0$, or $\lat\lab \varphi(c')$ belongs to $\Gamma_0$, or $\laf\lab \varphi(c')$ belongs to $\Gamma_0$, and $\laF\lab \varphi(c')$ does not belong to $\Gamma_0$. This proves that $\Gamma_0$ satisfies clause~11 of Definition~\ref{hintikka}. Finally, if  $\laF\lab\forall x\varphi \in \Gamma_0$ then, since $\mathcal{F}$ is a finished systematic tableau and $\theta$ is an open branch, $\laF\lab\forall x\varphi$ was used, as observed in Remark~\ref{rem-finished}, and so  $\laF\lab\varphi(c) \in \Gamma_0$  for some $c \in C \cup \bar C$. From this, $\Gamma_0$ also satisfies clause~12 of Definition~\ref{hintikka}. That is, $\Gamma_0$ is a Hintikka set for \tm\ in the universe $C \cup \bar C$. 
\end{proof}

\begin{coro} Let $\laL\lab\varphi$  be a signed formula over $\Theta$, and let $\bar C$ be as in Definition~\ref{systematic}. Let $\theta$ be an open branch of a finished systematic tableau in \tms\ for $\laL\lab\varphi$, and let $\Gamma_0$ be the set of signed formulas occurring in  $\theta$   (so, marked signed formulas as $\lat\lab\forall x\psi\lab[c]$  or $\laf\lab\forall x\psi\lab[c]$ are not included in $\Gamma_0$). Then, there  is a structure $\mathfrak{A}$ for \tms\ over $\Theta(\bar C)$ and a valuation $\bar v$ over it such that $\laL\lab\gamma$ is true in $\bar v$  for every $\laL\lab\gamma \in \Gamma_0$.
\end{coro}
\begin{proof}
	It is an immediate consequence of Proposition~\ref{open-hintikka} and Theorem~\ref{hintikka-lemma}. 
\end{proof}

\noindent Observe that, as a consequence of the definitions, if $\varphi$ is provable by tableaux in \tms\ then the systematic tableau for $\laL\lab\varphi$ must close after a finite number of steps, for $\laL \in \{\laF, \laf\}$ (and the converse is also true, of course). This produces the following:

\begin{teo} [Completeness of tableaux for \tms] Let $\Gamma \cup\{\varphi\}$ be a finite set of closed formulas over $\Theta$. If $\Gamma \models_\tms \varphi$ \ then \ $\Gamma \models_{\mathcal{T}(\tms)} \varphi$.
\end{teo}
\begin{proof}
By definition of $\models_{\mathcal{T}(\tms)}$, and since $\models_\tms$ satisfies the deduction metatheorem for sentences, it is enough to prove the result for $\Gamma=\emptyset$. Thus, let $\F$ be a finished systematic tableau in \tms\ for $\laL\lab\varphi$, where $\laL \in \{\laF, \laf\}$. If $\F$ has an open branch $\theta$ then the set $\Gamma_0$ of signed formulas occurring in  $\theta$  is simultaneously satisfiable by a valuation $\bar v$ over a first-order structure $\mathfrak{A}$ for \tms. In particular, $\laL\lab\varphi$ is true in $\bar v$, meaning that $\bar v(\varphi) \in \{\laF, \laf\}$. That is, $\not\models_\tms \varphi$. From this, if $\models_\tms \varphi$ then the systematic tableau for $\laL\lab\varphi$ closes in a finite number of steps, for any $\laL \in \{\laF, \laf\}$. Indeed, since every branch of  $\F$ is closed, every branch of $\F$ is finite. Then, by K\"onig's lemma, $\F$  must be finite. That is, $\varphi$ is provable by tableaux in \tms.
\end{proof}

\subsection{Tableaux for {\bf S4m}, {\sqms}, {\bf S5m} and {\scms} } \label{tableaux-etc}

The rules for {\bf S4m} are the same for \tm,  except for the operator $\Box$. In this case, the rules are as follows:

$$
\begin{array}{llll}
\displaystyle \frac{\laT\lab  \Box \varphi}{\laT\lab  \varphi} & \hspace{6mm} \displaystyle \frac{\lat\lab  \Box \varphi}{\times}  & \hspace{6mm} \displaystyle \frac{\laf\lab \Box \varphi}{\lat\lab  \varphi \mid \laf\lab  \varphi \mid \laF\lab  \varphi} & \hspace{6mm} \displaystyle \frac{\laF\lab  \Box \varphi}{\lat\lab  \varphi \mid \laf\lab  \varphi \mid \laF\lab  \varphi} \\[2mm]
&&\\[2mm]
\end{array}
$$

\noi where the symbol $\times$ in the rule for $\lat\lab  \Box \varphi$ indicates that the branch closes. This rule reflects the fact that a formula of the type $\Box\varphi$ cannot receive the value \lat\ in 
$\mathcal{M}({\bf S4m})$.

If we add the quantified rules for \tms\ to the set of rules for {\bf S4m}, we obtain the tableau rules for {\sqms}.

The rules for {\bf S5m} are also the same for \tm,  except for the operator $\Box$. In this case, the rules are as follows:

$$
\begin{array}{llll}
\displaystyle \frac{\laT\lab  \Box \varphi}{\laT\lab  \varphi} & \hspace{6mm} \displaystyle \frac{\lat\lab  \Box \varphi}{\times}  & \hspace{6mm} \displaystyle \frac{\laf\lab \Box \varphi}{\times} & \hspace{6mm} \displaystyle \frac{\laF\lab  \Box \varphi}{\lat\lab  \varphi \mid \laf\lab  \varphi \mid \laF\lab  \varphi} \\[2mm]
&&\\[2mm]
\end{array}
$$

\noi Once again, the symbol  $\times$ indicates that the branch closes in the corresponding rule. Indeed, in $\mathcal{M}({\bf S5m})$ it is impossible for a formula of the type $\Box\varphi$  to receive the value \lat\ or the value \laf. 

If we add the quantified rules for \tms\ to the set of rules for {\bf S5m}, we obtain the tableau rules for {\scms}.

As observed above, the proof of the following result will be omitted here, but it can be obtained by slight modifications of the one presented for  \tms

\begin{teo} [Soundness and Completeness of tableaux for \sqms\ and \scms] \ \\
Let ${\bf L} \in \{\sqms, \scms\}$ and let $\Gamma \cup\{\varphi\}$ be a finite set of closed formulas over $\Theta$. Then: $\Gamma \models_{\bf L} \varphi$ \ if and only if \ $\Gamma \models_{\mathcal{T}({\bf L})} \varphi$.
\end{teo}

\section*{Final Remarks} \label{finalsect}

In this paper we introduce analytic tableaux for several (propositional and quantified) non-normal modal logics with non-deterministic semantics. First, we introduce tableau systems for the non-deterministic propositional modal systems \tm, \sqm, and \scm . The reader can easily check that all rules are decidable, since the tableaux trees never go into an infinite loop. We believe that such result can be easily extended to six-valued Ivlev-like systems such as the deontic {\bf Dm} (see \cite{con:cer:per:15} and \cite{con:cer:per:17}), or even to eight-valued systems such as {\bf Km} (see \cite{con:cer:per:19}).

Comparing the tableau systems of the Kripkean modal logics with the ones for the respective Ivlev-like systems, the latter seem to have a non-negligible advantage. While the former require the use of rules between trees (see \cite[Chap.~2]{fit:med:98}), the tableaux for \tm, \sqm, and \scm\ use rules only for the branches. This shows that the algorithmic complexity of this proof  method grows only as a function of the size of the formulas, as occurs in classical logic or finite many-valued logics.
This kind of result is crucial when thinking about computational applications for these logics.


The tableau systems for propositional logics are extended to the quantified versions of the systems mentioned above, namely \tms, \sqms\ and \scms. We know that {\bf CL$^*$} is undecidable. This is easy to check: considering the tableaux rules for classical predicate logic, we verify that, for instance, the formula $\forall x \exists y Rxy \to \exists x \forall y Rxy$ cannot be refuted. Indeed, when we try to finish a tableau tree to refute this formula, we can see that the rules lead us to an infinite loop. Since the three quantified modal systems presented here are extensions of classical logic, they are undecidable as well.

Although {\bf CL$^*$} is not decidable, the monadic fragment of it is decidable (see, for example, \cite[Chap.~21]{boo:etal:2002}). This result also holds for the ${\bf G}_n$ hierarchy of Gödel $n$-valued logics, as proved in \cite{Baaz:etal:07}, and it seems to hold for any monadic fragment of multivalued logic.

It is natural to ask whether the same result would hold for the monadic fragment of some first-order modal systems, in particular the ones presented here. Kripke proved in \cite{kri:62} a pretty strong result: any monadic fragment of an {\bf S5*} subsystem is undecidable. But \tms, \sqms\ and \scms\ are not subsystem of {\bf S5*}, since \axNBF\ does not hold in {\bf S5*}. Whether or not the monadic fragments of the three modal systems studied here are decidable is still an open question, although we have strong reasons to believe that they are not.

It seems that there is a very big difference between classical logic and finite-valued logics, on the one hand, and relational semantics and non-deterministic semantics, on the other. Indeed, since classical and finite-valued logics are extensional, they are unable to semantically express intensional operators. This could suggests a rather strong result: any monadic fragment of intensional semantics would be undecidable. In any case, concerning modal logic, these results lead us to agree with Kripke, who said that ``in the domain of modal logic, decidable monadic systems simply do not arise''. This important question deserves further analysis.

Concerning the full (normal) version of the propositional modal systems discussed here, it should be noticed that, recently, Gr\"atz has modified the four-valued Nmatrix semantics with level valuations for {\bf S4} introduced by Kearns in~\cite{kear:81}, obtaining so a three-valued Nmatrix semantics with level valuations in which the criteria for choosing the level valuations is effective (see~\cite{gratz:21b}). This constitutes a novel decision procedure for modal systems {\bf T} and {\bf S4}, overcoming a criticism to the original method of Kearns we made in~\cite[Section~4]{con:cer:per:17}, precisely concerning its status as a decision procedure. In turn, in~\cite{paw:larosa:21} it was proposed a new four-valued Nmatrix for a weaker version of Ivlev's {\bf Tm} called ${\bf T^-}$, which is axiomatized just by removing the Necessitation rule (NEC) from the standard axiomatization of  {\bf T}. They consider 16 axiomatic extensions of  ${\bf T^-}$ as well as the corresponding modification in the basic four-valued Nmatrix, recovering so (by adding (NEC), at the axiomatic level, and by considering level valuations, on the semantical side) the systems {\bf T}, {\bf TB}, {\bf S4}, and {\bf S5}, together with two other new systems.


To summarize, we believe that the many results for Ivlev-like modal systems presented in the literature involving Nmatrix semantics, as well as the recent results on Kearns' Nmatrix semantics with level valuations for normal modal systems above mentioned, open up concrete and exciting new perspectives for the study of modal logics in general.

\bibliographystyle{apalike}

\end{document}